\documentclass[11pt,a4paper,reqno]{amsart}

\usepackage[T1]{fontenc}
\usepackage[utf8]{inputenc}
\usepackage[british]{babel}
\usepackage{lmodern}
\usepackage{geometry}
\usepackage{amsmath,amssymb,mathtools,bm}
\usepackage{booktabs,array}
\usepackage{enumitem}
\usepackage{graphicx}
\usepackage{placeins}
\usepackage{microtype}
\usepackage[numbers,sort&compress]{natbib}
\usepackage[hidelinks]{hyperref}
\usepackage{xurl}

\hypersetup{
  pdftitle={Crouzeix--Raviart--Marini realisation of computable a priori L2-error bounds on anisotropic meshes},
  pdfauthor={Hiroki Ishizaka},
  pdfsubject={Computable finite element error bounds on anisotropic meshes},
  pdfkeywords={computable a priori error bound, anisotropic finite element method, Crouzeix--Raviart method, Raviart--Thomas flux, Marini relation, non-convex domain, graded mesh}
}

\numberwithin{equation}{section}

\theoremstyle{plain}
\newtheorem{theorem}{Theorem}[section]
\newtheorem{lemma}[theorem]{Lemma}
\newtheorem{proposition}[theorem]{Proposition}
\newtheorem{corollary}[theorem]{Corollary}

\theoremstyle{definition}

\newtheorem{assumption}[theorem]{Assumption}

\theoremstyle{remark}
\newtheorem{remark}[theorem]{Remark}

\newcommand{\R}{\mathbb{R}}

\newcommand{\cL}{\mathcal{L}}

\newcommand{\RT}{\mathrm{RT}}
\newcommand{\CR}{\mathrm{CR}}

\newcommand{\Pone}{\mathbb{P}^{1}}
\newcommand{\Pzero}{\mathbb{P}^{0}}
\newcommand{\divop}{\operatorname{div}}

\newcommand{\PiZero}{\Pi_h^0}
\newcommand{\gradH}{\nabla_h}

\title[CR--Marini realisation of computable $L^2$-error bounds]
{Crouzeix--Raviart--Marini realisation of computable a priori \texorpdfstring{$L^2$}{L2}-error bounds on anisotropic meshes}

\author{Hiroki Ishizaka}
\address{Team FEM, Matsuyama, Japan}
\email{h.ishizaka005@gmail.com}
\urladdr{https://teamfem.github.io/hiroki\_ishizaka/}

\date{}



\begin{document}

\begin{abstract}
We consider conforming $\Pone$ and lowest-order Crouzeix--Raviart (CR) approximations of the Dirichlet Poisson problem in two and three dimensions. We obtain computable a priori $L^2$-error bounds on anisotropic simplicial meshes without a global $H^2$-regularity assumption. For piecewise constant data, the Marini relation gives an equilibrated lowest-order Raviart--Thomas (RT) flux from the scalar CR solution. The square of the resulting Marini constant is the largest generalised eigenvalue of a problem involving only scalar conforming and CR matrices. Combined with the elementwise Poincar\'e constant, it controls both the conforming and CR errors. The CR $L^2$ estimate uses positivity of the discrete CR--conforming gap operator rather than a nonconforming Aubin--Nitsche argument. We also derive an exact decomposition of the Marini defect into the CR--conforming energy gap and an explicit geometric term with no aspect-ratio factor. On the algebraically graded L-shaped mesh family considered here, $q>3/2$ yields $\kappa_{M,h}=O(h)$ and computable $O(h^2)$ $L^2$-operator bounds.
\end{abstract}

\maketitle

\section{Introduction}\label{sec:introduction}
For conforming $\Pone$ elements, the usual Aubin--Nitsche estimate relies on elliptic regularity. This is inconvenient on non-convex polygons and polyhedra, where $L^2$ data need not produce an $H^2$ solution. Computable $L^2$ estimates that avoid this regularity assumption were developed in \citep{KinoshitaHashimotoNakao2009,LiuOishi2013,TakayasuLiuOishi2014}. Here, we treat the lowest-order CR--RT pair and use the Marini relation to evaluate the relevant constant through scalar finite element problems.

Let $d\in\{2,3\}$ and let $\mathbb T_h$ be a conforming simplicial mesh of $\Omega$. For $k\in\mathbb N_0:=\mathbb N\cup\{0\}$, let $\mathbb P^k(T)$ denote the polynomials on $T$ of total degree at most $k$, and we set
\begin{align*}
  \mathbb P^k(\mathbb T_h)
  :=\{v_h\in L^2(\Omega):v_h|_T\in\mathbb P^k(T)
       \ \forall T\in\mathbb T_h\},
  \quad
  X_h:=\Pzero(\mathbb T_h).
\end{align*}
For $g_h\in X_h$, let $z_h^{\CR}(g_h)$ be the auxiliary CR solution and we set
\begin{align}\label{eq:intro:Marini}
  \bm{\sigma}_h^M(g_h)|_T
  :=\gradH z_h^{\CR}(g_h)|_T
    -\frac{g_h|_T}{d}(\bm x-\bm x_T).
\end{align}
Here, $\bm x_T$ is the barycentre of $T$. Formula~\eqref{eq:intro:Marini} is the classical Marini relation in two dimensions \citep{Marini1985}; the three-dimensional version is given in \citet[Lemma~10]{IshizakaKobayashiTsuchiya2021}. For piecewise constant data, \eqref{eq:intro:Marini} replaces the mixed RT reconstruction in the hypercircle argument by a scalar CR solve. Related uses of the Marini relation in nonconforming and equilibrated-flux analysis can be found in \citep{Braess2009,LiuKikuchi2018,ErnVohralik2013}. In this paper it is used both for the a priori $L^2$ constant and for the CR error bound.

The regularity-free hypercircle estimate itself is not new, and neither is the Marini relation. The contribution here is the way these two ingredients are combined for the lowest-order CR--RT pair. We obtain an exact decomposition of the Marini defect into the CR--conforming energy gap and an explicit geometric term, characterise $\kappa_{M,h}^2$ by a scalar generalised eigenvalue problem, and derive the CR $L^2$ bound from positivity of the gap operator. The same representation is then used to study anisotropic and graded mesh families, where it also identifies which part of the defect controls a loss of rate.

With $z_h^c(g_h)$ the conforming auxiliary solution, we define
\begin{align}\label{eq:intro:kappaM}
  \kappa_{M,h}
  :=\sup_{0\ne g_h\in X_h}
  \frac{\|\bm{\sigma}_h^M(g_h)-\nabla z_h^c(g_h)\|_{L^2(\Omega)^d}}
       {\|g_h\|_{L^2(\Omega)}}.
\end{align}
Writing $g_T:=g_h|_T\in\R$ and $\mu_T:=\int_T|\bm x-\bm x_T|^2\,dx$, we prove the exact decomposition
\begin{align}\label{eq:intro:defect-decomp}
  \|\bm{\sigma}_h^M(g_h)-\nabla z_h^c(g_h)\|_{L^2(\Omega)^d}^2
  =|z_h^{\CR}(g_h)-z_h^c(g_h)|_{H^1(\mathbb T_h)}^2
   +\frac1{d^2}\sum_{T\in\mathbb T_h}\mu_T|g_T|^2.
\end{align}
To state the resulting finite-dimensional formula, let $\mathbf g$ denote the coefficient vector of $g_h$ with respect to the element-characteristic basis of $X_h$. We define the matrices $M_0$, $G_{\rm gap}$, and $D_M$ through the following quadratic-form identities:
\begin{align*}
  \mathbf g^\top M_0\mathbf g
  &:=\|g_h\|_{L^2(\Omega)}^2,\\
  \mathbf g^\top G_{\rm gap}\mathbf g
  &:=|z_h^{\CR}(g_h)-z_h^c(g_h)|_{H^1(\mathbb T_h)}^2,\\
  \mathbf g^\top D_M\mathbf g
  &:=\frac1{d^2}\sum_{T\in\mathbb T_h}\mu_T|g_T|^2.
\end{align*}
Then, \eqref{eq:intro:defect-decomp} gives the exact scalar spectral formula
\begin{align*}
  \kappa_{M,h}^2=\lambda_{\max}(G_{\rm gap}+D_M,M_0).
\end{align*}
No RT basis or mixed saddle-point system is needed for this eigenvalue problem: only the scalar CR and conforming matrices and the diagonal geometric term $D_M$ enter.

Let $C_0(T)$ be the optimal constant in the mean-zero Poincar\'e inequality on $T$, and set
\begin{align*}
  C_{0,h}:=\max_{T\in\mathbb T_h}C_0(T),
  \quad
  M_h:=\sqrt{C_{0,h}^2+\kappa_{M,h}^2}.
\end{align*}
For the Poisson solution $u$, its conforming approximation $u_h$, and its lowest-order CR approximation $u_h^{\CR}$, the same mesh-dependent quantity controls both $L^2$ errors:
\begin{align}\label{eq:intro:L2-main}
  \|u-u_h\|_{L^2(\Omega)}
  &\le M_h^2\|f\|_{L^2(\Omega)},\\
  \|u-u_h^{\CR}\|_{L^2(\Omega)}
  &\le M_h^2\|f\|_{L^2(\Omega)}.
\end{align}
Furthermore,
\begin{align*}
  |u-u_h^{\CR}|_{H^1(\mathbb T_h)}
  \le 2M_h\|f\|_{L^2(\Omega)}.
\end{align*}
If $h:=\max_{T\in\mathbb T_h}\operatorname{diam}(T)$, the Payne--Weinberger estimate $C_{0,h}\le h/\pi$ gives the fully computable replacement
\begin{align*}
  \widehat M_h:=\sqrt{(h/\pi)^2+\kappa_{M,h}^2}.
\end{align*}
Replacing $M_h$ by $\widehat M_h$ gives the explicit majorants used in Section~\ref{sec:numerics}. The proof of the CR $L^2$ estimate uses positivity of the discrete CR--conforming gap operator and does not use a nonconforming Aubin--Nitsche argument.

The CR--RT reconstruction also appears in the semilinear certification problem of \citep{Ishizaka2026Marini}, where it is used for an equilibrated Newton--Kantorovich residual. The present paper concerns instead the linear Poisson approximation constant and its conforming and CR consequences.

For mesh-family rates we use interpolation estimates written in physical directions. In two dimensions, the directional Lagrange and RT bounds of \citet{Ishizaka2025Interpolation} and \citet{Ishizaka2022RT}, together with the corner decomposition, give $\kappa_{M,h}=O(h)$ on the algebraically graded L-shaped family when $q>3/2$; the resulting computable $L^2$-operator bounds are $O(h^2)$. In three dimensions we only give an anisotropic comparison under local $W^{2,p}$ assumptions, with $p>2$ for the first-order $H^1$ Lagrange estimate. These assumptions are used for rates, not for the single-mesh construction or the defect identity. Sharp rates on general graded tetrahedral meshes with corner and edge singularities are left open.

Sections~\ref{sec:setting}--\ref{sec:matrix} contain the finite element setting, the CR--Marini reconstruction, the $L^2$ estimates, and the scalar eigenvalue formula. Anisotropic rates and corner grading are discussed in Section~\ref{sec:anisotropic}; the numerical results are in Section~\ref{sec:numerics}.

\section{Poisson problem and finite element setting}\label{sec:setting}
Let $d\in\{2,3\}$ and let $\Omega\subset\R^d$ be a bounded polyhedral Lipschitz domain. Throughout, $(\cdot,\cdot)$ denotes the $L^2(\Omega)$ inner product of scalar- or vector-valued functions, as appropriate, while $(\cdot,\cdot)_T$ denotes the corresponding inner product over an element $T\in\mathbb T_h$. We set
\begin{align*}
  V:=H_0^1(\Omega),
  \quad
  a(v,w):=(\nabla v,\nabla w)_{L^2(\Omega)^d},
  \quad
  |v|_{H^1(\Omega)} := \|\nabla v\|_{L^2(\Omega)^d}.
\end{align*}
For $f\in L^2(\Omega)$, let $u\in V$ be the unique weak solution of
\begin{align}\label{eq:poisson}
  a(u,v)=(f,v)_{L^2(\Omega)}
  \quad\forall v\in V.
\end{align}
Equivalently, $-\varDelta u=f$ in the distributional sense. In particular,
\begin{align}\label{eq:graphspace}
  u\in\{v\in H_0^1(\Omega):\varDelta v\in L^2(\Omega)\},
\end{align}
without any implication that $u\in H^2(\Omega)$.

Let $\mathbb{T}_h=\{T\}$ be a finite conforming simplicial triangulation of $\overline\Omega$ into non-degenerate closed $d$-simplices. Thus, $\overline\Omega=\bigcup_{T\in\mathbb T_h}T$, and the intersection of any two distinct elements is either empty or a common subsimplex. We set $h_T:=\operatorname{diam}(T)$ and $h:=\max_{T\in\mathbb T_h}h_T$. No minimum-angle, maximum-angle, or quasi-uniformity assumption is imposed. Let
\begin{align*}
  V_h^c:=\{v_h\in V: \ v_h|_T\in \Pone(T) \ \forall T\in \mathbb{T}_h\}
\end{align*}
be the conforming piecewise linear finite element space. The Galerkin solution $u_h\in V_h^c$ is defined as
\begin{align}\label{eq:P1}
  a(u_h,v_h)=(f,v_h)
  \quad\forall v_h\in V_h^c.
\end{align}
Let $P_h:V\to V_h^c$ denote the Ritz projection. Then, $u_h=P_hu$ and
\begin{align}\label{eq:galerkin-orthogonality}
  a(u-u_h,v_h)=0\quad\forall v_h\in V_h^c.
\end{align}

Recall that $X_h$ denotes the space of elementwise constant functions. We define the $L^2$-orthogonal projection $\PiZero:L^2(\Omega)\to X_h$ as
\begin{align*}
  (\PiZero g)|_T=: \Pi_T^0(g|_T)
  :=\frac1{|T|_d}\int_Tg\,dx
  \quad\forall T\in\mathbb T_h.
\end{align*}
Here, $|\cdot|_d$ denotes the $d$-dimensional Hausdorff measure. For each simplex $T$, we define the optimal mean-zero Poincar\'e constant as
\begin{align}\label{eq:C0T}
  C_0(T)
  :=\sup_{\substack{v\in H^1(T),\ \int_Tv=0\\v\ne0}}
       \frac{\|v\|_{L^2(T)}}{\|\nabla v\|_{L^2(T)^d}},
  \quad
  C_{0,h}:=\max_{T\in\mathbb{T}_h}C_0(T).
\end{align}
Because any simplex is convex, the Payne--Weinberger inequality \citep{PayneWeinberger1960} yields
\begin{align}\label{eq:C0h-diam}
  C_0(T)\le\frac{h_T}{\pi},
  \quad
  C_{0,h}\le\frac{h}{\pi}.
\end{align}
The convexity used in \eqref{eq:C0h-diam} is only the local convexity of each simplex and places no convexity assumption on $\Omega$. Although sharper simplex-specific bounds are available \citep{LaugesenSiudeja2010}, Section~\ref{sec:numerics} uses $h/\pi$ for simplicity and reproducibility.

We equip $H^{-1}(\Omega)=V^*$ with the energy-dual norm
\begin{align*}
  \|\ell\|_{H^{-1}(\Omega)}
  :=\sup_{0\ne v\in V}
  \frac{|\langle\ell,v\rangle|}{|v|_{H^1(\Omega)}}.
\end{align*}

\begin{lemma}[Elementwise projection in the dual norm]\label{lem:osc}
For any $g\in L^2(\Omega)$,
\begin{align}\label{eq:osc-Hminus1}
  \|g-\PiZero g\|_{H^{-1}(\Omega)}
  \le C_{0,h}\|g-\PiZero g\|_{L^2(\Omega)}.
\end{align}
\end{lemma}

\begin{proof}
Let $v\in V$.  Since $g-\PiZero g$ has zero mean on any element,
\begin{align*}
  (g-\PiZero g,v)
  =\sum_{T\in\mathbb{T}_h}(g-\PiZero g,v-\Pi_T^0v)_T.
\end{align*}
The definition of $C_0(T)$, followed by Cauchy--Schwarz over the mesh, gives
\begin{align*}
  |(g-\PiZero g,v)|
  \le C_{0,h}\|g-\PiZero g\|_{L^2(\Omega)}|v|_{H^1(\Omega)}.
\end{align*}
Taking the supremum over $|v|_{H^1(\Omega)}=1$ proves the result.
\end{proof}

\section{CR--Marini equilibrated fluxes}\label{sec:marini}
Set
\begin{align*}
H^1(\mathbb{T}_h)
&:= \left\{ \varphi \in L^2(\Omega):
\ \varphi|_{T} \in H^1(T) \ \forall T \in \mathbb{T}_h \right\}.
\end{align*}
For any $v \in H^1(\mathbb{T}_h)$, we write $\nabla_h v$ for the broken gradient, defined by $(\nabla_h v)|_T := \nabla(v|_T)$ for each element $T\in\mathbb T_h$. If $v\in H^1(\Omega)$, then $\nabla_h v=\nabla v$ a.e.\ in $\Omega$.
For $v\in H^1(\mathbb T_h)$, we use the broken seminorm
\begin{align*}
  |v|_{H^1(\mathbb T_h)}
  :=\|\nabla_h v\|_{L^2(\Omega)^d}.
\end{align*}

Let $\mathcal{F}_h^i$ and $\mathcal{F}_h^{\partial}$ denote the sets of interior and boundary faces, respectively, and set $\mathcal{F}_h:=\mathcal{F}_h^i\cup\mathcal{F}_h^{\partial}$. For each $F=\partial T_+\cap\partial T_-\in\mathcal F_h^i$, choose a unit normal $\bm n_F$ pointing from $T_+$ to $T_-$. For $\varphi\in H^1(\mathbb T_h)$, we define
\begin{align*}
  [\![\varphi]\!]_F
  :=\varphi|_{T_+}-\varphi|_{T_-}.
\end{align*}
If $F=\partial T\cap\partial\Omega$, set $[\![\varphi]\!]_F:=\varphi|_T$. For $\bm v\in H^1(\mathbb T_h)^d$ and $F\in\mathcal F_h^i$, we define the normal jump as
\begin{align*}
  [\![\bm v\cdot\bm n]\!]_F
  :=\bm v|_{T_+}\cdot\bm n_F-\bm v|_{T_-}\cdot\bm n_F.
\end{align*}

Let $V_{h,0}^{\CR}$ denote the lowest-order CR space \citep{CrouzeixRaviart1973} with vanishing boundary face means. Thus,
\begin{align*}
\displaystyle
V_{h,0}^{\CR} &:=  \left \{ \varphi_h \in \mathbb{P}^1(\mathbb{T}_h): \  \int_F [\![ \varphi_h ]\!] ds = 0 \ \forall F \in \mathcal{F}_h \right \}.
\end{align*}
For $g_h\in X_h$, let $z_h^{\CR}(g_h)\in V_{h,0}^{\CR}$ solve
\begin{align}\label{eq:CR-rhs}
  (\gradH z_h^{\CR}(g_h),\gradH v_h)=(g_h,v_h)
  \quad\forall v_h\in V_{h,0}^{\CR}.
\end{align}
The broken seminorm $|\cdot|_{H^1(\mathbb T_h)}$ is a norm on $V_{h,0}^{\CR}$: a function with vanishing broken gradient is elementwise constant, the interior face-mean conditions identify these constants across the connected mesh, and the boundary face-mean conditions force the common constant to vanish. Therefore, \eqref{eq:CR-rhs} has a unique solution for any $g_h\in X_h$. For $T \in \mathbb{T}_h$, the local RT polynomial space is defined as
\begin{align*}
\displaystyle
\RT^0(T)
:=\{\bm a+b\bm x:\ \bm a\in\mathbb R^d,\ b\in\mathbb R\}
=\mathbb{P}^0(T)^d+\bm x\mathbb{P}^0(T).
\end{align*}
The globally normal-continuous lowest-order RT finite element space \citep{RaviartThomas1977} is defined as follows:
\begin{align*}
\displaystyle
V_h^{\RT} &:= \{ \bm v_h \in L^2(\Omega)^d: \  \bm v_h|_T \in \RT^0(T), \ \forall T \in \mathbb{T}_h, \  [\![ \bm v_h \cdot \bm n ]\!]_F = 0, \ \forall F \in \mathcal{F}_h^i \}.
\end{align*}
The normal-continuity condition implies $V_h^{\RT}\subset H(\divop;\Omega)$. Let $\bm x_T$ denote the barycentre of $T$. We define the Marini-type flux $\bm{\sigma}_h^M(g_h)$ elementwise as
\begin{align}\label{eq:Marini-flux}
  \bm{\sigma}_h^M(g_h)|_T
  :=\gradH z_h^{\CR}(g_h)|_T
    -\frac{g_h|_T}{d}(\bm x-\bm x_T).
\end{align}

In two dimensions, \eqref{eq:Marini-flux} is the classical Marini relation \citep{Marini1985}; the three-dimensional version is given in \citet[Lemma~10]{IshizakaKobayashiTsuchiya2021}. We include a short proof of the conformity and divergence properties for $d\in\{2,3\}$.

\begin{proposition}[Marini reconstruction]\label{prop:Marini}
For any $g_h\in X_h$,
\begin{align}\label{eq:Marini-properties}
  \bm{\sigma}_h^M(g_h)\in V_h^{\RT}\subset H(\divop;\Omega),
  \quad
  \divop\bm{\sigma}_h^M(g_h)=-g_h .
\end{align}
\end{proposition}

\begin{proof}
We set $g_T:=g_h|_T$. Because $\gradH z_h^{\CR}(g_h)|_T$ is constant and $\divop(\bm x-\bm x_T)=d$, definition \eqref{eq:Marini-flux} gives
\begin{align*}
  \bm{\sigma}_h^M(g_h)|_T\in\RT^0(T),
  \quad
  \divop\bm{\sigma}_h^M(g_h)|_T=-g_T.
\end{align*}
It remains to prove normal continuity across each interior face. Let $F=\partial T_+\cap\partial T_-\in\mathcal F_h^i$, let $\bm n_{T,F}$ denote the outward unit normal to $T$ on $F$, and let $\varphi_F\in V_{h,0}^{\CR}$ be the global CR basis function whose face average is one on $F$ and zero on any other face. For $T\in\{T_+,T_-\}$, the barycentric-coordinate formula for $\varphi_F|_T$ gives
\begin{align}\label{eq:CR-face-basis-identities}
  \nabla\varphi_F|_T
  =\frac{|F|_{d-1}}{|T|_d}\bm n_{T,F},
  \quad
  \int_T\varphi_F\,dx=\frac{|T|_d}{d+1}.
\end{align}
Furthermore, for any $\bm x\in F$,
\begin{align}\label{eq:barycentre-face-distance}
  (\bm x-\bm x_T)\cdot\bm n_{T,F}
  =\frac{d|T|_d}{(d+1)|F|_{d-1}}.
\end{align}
Indeed, the height of $T$ over $F$ is $d|T|_d/|F|_{d-1}$, and the barycentre has barycentric coordinate $1/(d+1)$ at the vertex opposite $F$. Testing \eqref{eq:CR-rhs} with $\varphi_F$ and using \eqref{eq:CR-face-basis-identities} yields
\begin{align*}
  |F|_{d-1}
  \sum_{T\in\{T_+,T_-\}}
  \gradH z_h^{\CR}(g_h)|_T\cdot\bm n_{T,F}
  =\frac1{d+1}
  \sum_{T\in\{T_+,T_-\}}g_T|T|_d.
\end{align*}
Combining this identity with \eqref{eq:Marini-flux} and \eqref{eq:barycentre-face-distance}, we obtain
\begin{align*}
  |F|_{d-1}
  \sum_{T\in\{T_+,T_-\}}
  \bm{\sigma}_h^M(g_h)|_T\cdot\bm n_{T,F}=0.
\end{align*}
Thus, the normal component is continuous across every interior face. Together with the local $\RT^0$ membership, this proves $\bm{\sigma}_h^M(g_h)\in V_h^{\RT}\subset H(\divop;\Omega)$, and the elementwise divergence identity proves $\divop\bm{\sigma}_h^M(g_h)=-g_h$.
\end{proof}

For the same function $g_h\in X_h$, let $z_h^c(g_h)\in V_h^c$ be the conforming solution
\begin{align}\label{eq:conforming-rhs}
  a(z_h^c(g_h),v_h)=(g_h,v_h)
  \quad \forall v_h\in V_h^c.
\end{align}
We define the central computable constant as
\begin{align}\label{eq:kappaM-def}
  \kappa_{M,h}
  :=\sup_{0\ne g_h\in X_h}
  \frac{\|\bm{\sigma}_h^M(g_h)-\nabla z_h^c(g_h)\|_{L^2(\Omega)^d}}
       {\|g_h\|_{L^2(\Omega)}}.
\end{align}

\begin{proposition}[Exact CR--Marini defect decomposition]\label{prop:defect-decomp}
For each simplex $T\in\mathbb T_h$ with vertices
$\bm x_1,\ldots,\bm x_{d+1}$, we define
\begin{align*}
  \mu_T:=\int_T |\bm x-\bm x_T|^2\,dx.
\end{align*}
It has the explicit representation
\begin{align}\label{eq:muT}
\mu_T
 =  \frac{|T|_d}{(d+1)^2(d+2)}
    \sum_{1\le i<j\le d+1}|\bm x_i-\bm x_j|^2.
\end{align}
For any $g_h\in X_h$, we write $g_T:=g_h|_T\in\R$ for each
$T\in\mathbb T_h$. Then,
\begin{align}
 \|\bm{\sigma}_h^M(g_h)-\nabla z_h^c(g_h)\|_{L^2(\Omega)^d}^2
 =|z_h^{\CR}(g_h)-z_h^c(g_h)|_{H^1(\mathbb{T}_h)}^2
   +\frac1{d^2}\sum_{T\in\mathbb{T}_h}\mu_T\,|g_T|^2,  \label{eq:defect-decomp}
\end{align}
and
\begin{align}\label{eq:energy-gap}
 |z_h^{\CR}(g_h)-z_h^c(g_h)|_{H^1(\mathbb{T}_h)}^2
 =|z_h^{\CR}(g_h)|_{H^1(\mathbb{T}_h)}^2-|z_h^c(g_h)|_{H^1(\Omega)}^2.
\end{align}
\end{proposition}

\begin{proof}
Because any conforming $\Pone$ function with homogeneous boundary values also satisfies the CR face-mean conditions, $V_h^c\subset V_{h,0}^{\CR}$. Therefore, for any $v_h^c\in V_h^c$,
\begin{align*}
 (\gradH z_h^{\CR}(g_h)-\nabla z_h^c(g_h),\nabla v_h^c)=0.
\end{align*}
Taking $v_h^c=z_h^c(g_h)$ gives
\begin{align*}
 (\gradH z_h^{\CR}(g_h),\nabla z_h^c(g_h))
 =\|\nabla z_h^c(g_h)\|_{L^2(\Omega)^d}^2.
\end{align*}
Therefore,
\begin{align*}
\begin{aligned}
 |z_h^{\CR}(g_h)-z_h^c(g_h)|_{H^1(\mathbb T_h)}^2
 &=
 |z_h^{\CR}(g_h)|_{H^1(\mathbb T_h)}^2
 +|z_h^c(g_h)|_{H^1(\Omega)}^2 
 -2(\gradH z_h^{\CR}(g_h),\nabla z_h^c(g_h)) \\
 &=
 |z_h^{\CR}(g_h)|_{H^1(\mathbb T_h)}^2
 -|z_h^c(g_h)|_{H^1(\Omega)}^2,
\end{aligned}
\end{align*}
which proves \eqref{eq:energy-gap}.

On each simplex, the vector $\gradH z_h^{\CR}(g_h)-\nabla z_h^c(g_h)$ is constant, whereas $\int_T(\bm x-\bm x_T)\,dx=0$. Consequently, the cross term in the square of \eqref{eq:Marini-flux} vanishes elementwise, and
\begin{align*}
 \|\bm{\sigma}_h^M(g_h)-\nabla z_h^c(g_h)\|_{L^2(\Omega)^d}^2
 =|z_h^{\CR}(g_h)-z_h^c(g_h)|_{H^1(\mathbb{T}_h)}^2
 +\frac1{d^2}\sum_T |g_T|^2\int_T|\bm x-\bm x_T|^2\,dx.
\end{align*}
For barycentric coordinates $\lambda_i$ on a $d$-simplex,
\begin{align*}
  \frac1{|T|_d}\int_T\lambda_i\lambda_j\,dx
  =\frac{1+\delta_{ij}}{(d+1)(d+2)}.
\end{align*}
Writing $\bm y_i:=\bm x_i-\bm x_T$, we have $\sum_{i=1}^{d+1}\bm y_i=0$. It follows that
\begin{align*}
 \int_T|\bm x-\bm x_T|^2\,dx
 =\frac{|T|_d}{(d+1)(d+2)}\sum_{i=1}^{d+1}|\bm x_i-\bm x_T|^2.
\end{align*}
Using
\begin{align*}
 \sum_{i=1}^{d+1}|\bm x_i-\bm x_T|^2
 =\frac1{d+1}\sum_{1\le i<j\le d+1}|\bm x_i-\bm x_j|^2
\end{align*}
gives \eqref{eq:muT}. By the definition of $\mu_T$, the preceding expansion is precisely \eqref{eq:defect-decomp}, which completes the proof.
\end{proof}

Both terms in \eqref{eq:defect-decomp} are defined on the physical mesh: $\mu_T$ contains the actual edge lengths of $T$, and the other term is the CR--conforming energy gap on the same mesh. Neither quantity involves an inradius or a shape-regularity factor.

\begin{lemma}[Hypercircle control for piecewise constant right-hand sides]\label{lem:hypercircle}
Let $g_h\in X_h$, and let $z(g_h)\in V$ solve $-\varDelta z(g_h)=g_h$. Then,
\begin{align}\label{eq:hypercircle-bound}
  |z(g_h)-z_h^c(g_h)|_{H^1(\Omega)}
  \le
  \|\bm{\sigma}_h^M(g_h)-\nabla z_h^c(g_h)\|_{L^2(\Omega)^d}
  \le \kappa_{M,h}\|g_h\|_{L^2(\Omega)}.
\end{align}
\end{lemma}

\begin{proof}
 We set $z:=z(g_h)$, $z_h:=z_h^c(g_h)$, and $\bm{\sigma}_h:=\bm{\sigma}_h^M(g_h)$. From Proposition~\ref{prop:Marini}, $\bm{\sigma}_h\in H(\divop;\Omega)$ and $\divop\bm{\sigma}_h=-g_h$. Since $z-z_h\in H_0^1(\Omega)$, the $H(\divop)$ Green formula and the weak formulation for $z$ yield
\begin{align*}
 (\bm{\sigma}_h,\nabla(z-z_h))
 =-(\divop\bm{\sigma}_h,z-z_h)
 =(g_h,z-z_h)
 =(\nabla z,\nabla(z-z_h)).
\end{align*}
Therefore,
\begin{align*}
 (\bm{\sigma}_h-\nabla z,\nabla z-\nabla z_h)=0.
\end{align*}
Thus,
\begin{align*}
 \|\bm{\sigma}_h-\nabla z_h\|_{L^2(\Omega)^d}^2
 =
 \|\nabla z-\nabla z_h\|_{L^2(\Omega)^d}^2
 +\|\bm{\sigma}_h-\nabla z\|_{L^2(\Omega)^d}^2,
\end{align*}
which is the Prager--Synge identity in the present setting \citep{PragerSynge1947}. Because the second term on the right-hand side is non-negative, taking square roots gives
\begin{align*}
 |z-z_h|_{H^1(\Omega)}
 \le
 \|\bm{\sigma}_h-\nabla z_h\|_{L^2(\Omega)^d}.
\end{align*}
For $g_h\ne0$, the definition of $\kappa_{M,h}$ gives
\begin{align*}
 \|\bm{\sigma}_h-\nabla z_h\|_{L^2(\Omega)^d}
 \le
 \kappa_{M,h}\|g_h\|_{L^2(\Omega)}.
\end{align*}
For $g_h=0$, this inequality is trivial. Thus, \eqref{eq:hypercircle-bound} follows.
\end{proof}

\begin{lemma}[Best equilibrated RT approximation]
\label{lem:best-equilibrated}
For $g_h\in X_h$, let $z=z(g_h)\in V$ solve
\begin{align*}
-\varDelta z=g_h
\quad\text{in }\Omega,
\quad
z=0
\quad\text{on }\partial\Omega,
\end{align*}
and we define
\begin{align*}
\mathcal Q_h(g_h)
:=
\{\bm{\tau}_h\in V_h^{\RT}: \
\divop\bm{\tau}_h=-g_h\}.
\end{align*}
Then, $\bm{\sigma}_h^M(g_h)$ is the unique minimiser of the equilibrated RT approximation problem and
\begin{align}\label{eq:best-equilibrated}
\|\nabla z-\bm{\sigma}_h^M(g_h)\|_{L^2(\Omega)^d}
=
\min_{\bm{\tau}_h\in\mathcal Q_h(g_h)}
\|\nabla z-\bm{\tau}_h\|_{L^2(\Omega)^d}.
\end{align}
\end{lemma}

\begin{proof}
From Proposition~\ref{prop:Marini}, $\bm{\sigma}_h^M(g_h)\in\mathcal Q_h(g_h)$. Let $\bm{\tau}_h\in\mathcal Q_h(g_h)$ and we set
\begin{align*}
\bm{\delta}_h:=\bm{\tau}_h-\bm{\sigma}_h^M(g_h).
\end{align*}
Then, $\bm{\delta}_h \in V_h^{\RT}$ and $\divop\bm{\delta}_h=0$. Because an $\RT^0$ field with zero divergence is constant on each element, $\bm{\delta}_h|_T$ is constant for any $T\in\mathbb T_h$. Since $z\in H_0^1(\Omega)$, the $H(\divop)$ Green formula gives
\begin{align*}
(\nabla z,\bm{\delta}_h)
=-(z,\divop\bm{\delta}_h)=0.
\end{align*}
Furthermore, using the Marini representation and
$\int_T(\bm x-\bm x_T)\,dx=0$,
\begin{align*}
(\bm{\sigma}_h^M(g_h),\bm{\delta}_h)
=
(\gradH z_h^{\CR}(g_h),\bm{\delta}_h).
\end{align*}
Elementwise integration by parts yields
\begin{align*}
(\gradH z_h^{\CR}(g_h),\bm{\delta}_h)
=
\sum_{T\in\mathbb T_h}
\int_{\partial T}
z_h^{\CR}(g_h)\,\bm{\delta}_h\cdot \bm n_T\,ds.
\end{align*}
On any interior face, $\bm{\delta}_h\cdot \bm n_F$ is single-valued and constant, whereas the CR jump has zero face mean. Therefore, all interior face contributions vanish. On boundary faces, the same conclusion follows from the homogeneous CR face-mean condition. Thus,
\begin{align*}
(\bm{\sigma}_h^M(g_h),\bm{\delta}_h)=0.
\end{align*}
Consequently,
\begin{align*}
(\nabla z-\bm{\sigma}_h^M(g_h),\bm{\delta}_h)=0.
\end{align*}
Because $\bm{\tau}_h=\bm{\sigma}_h^M(g_h)+\bm{\delta}_h$ and $(\nabla z-\bm{\sigma}_h^M(g_h),\bm{\delta}_h)=0$, the Pythagorean identity gives
\begin{align*}
\|\nabla z-\bm{\tau}_h\|_{L^2(\Omega)^d}^2
=
\|\nabla z-\bm{\sigma}_h^M(g_h)\|_{L^2(\Omega)^d}^2
+\|\bm{\delta}_h\|_{L^2(\Omega)^d}^2.
\end{align*}
As $\bm{\tau}_h\in\mathcal Q_h(g_h)$ was arbitrary, this proves \eqref{eq:best-equilibrated}. Furthermore, equality is possible only when $\|\bm{\delta}_h\|_{L^2(\Omega)^d}=0$, that is, when $\bm{\tau}_h=\bm{\sigma}_h^M(g_h)$; hence the minimiser is unique.
\end{proof}

\begin{corollary}[Joint hypercircle minimisation]\label{cor:joint-hypercircle}
For any $g_h\in X_h$, let $z=z(g_h)$ be the weak solution above. Then,
\begin{align}\label{eq:joint-hypercircle}
&\min_{v_h\in V_h^c}\ \min_{\bm{\tau}_h\in\mathcal Q_h(g_h)}
\|\bm{\tau}_h-\nabla v_h\|_{L^2(\Omega)^d} =
\|\bm{\sigma}_h^M(g_h)-\nabla z_h^c(g_h)\|_{L^2(\Omega)^d},
\end{align}
The minimum is attained uniquely when
\begin{align*}
v_h=z_h^c(g_h),
\quad
\bm{\tau}_h=\bm{\sigma}_h^M(g_h).
\end{align*}
Consequently,
\begin{align}\label{eq:kappa-joint-min}
\kappa_{M,h}
=\sup_{0\ne g_h\in X_h}
\frac{1}{\|g_h\|_{L^2(\Omega)}}
\min_{v_h\in V_h^c}\ \min_{\bm{\tau}_h\in\mathcal Q_h(g_h)}
\|\bm{\tau}_h-\nabla v_h\|_{L^2(\Omega)^d}.
\end{align}
\end{corollary}

\begin{proof}
Let $v_h\in V_h^c$ and $\bm{\tau}_h\in\mathcal Q_h(g_h)$ be arbitrary. Because $z-v_h\in H_0^1(\Omega)$, the $H(\divop)$ Green formula and $\divop\bm{\tau}_h=-g_h$ give
\begin{align*}
(\bm{\tau}_h,\nabla(z-v_h))_{L^2(\Omega)^d}
=
-(\divop\bm{\tau}_h,z-v_h)_{L^2(\Omega)}
=
(g_h,z-v_h)_{L^2(\Omega)}.
\end{align*}
On the other hand, the weak formulation for $z$ yields
\begin{align*}
(\nabla z,\nabla(z-v_h))_{L^2(\Omega)^d}
=
(g_h,z-v_h)_{L^2(\Omega)}.
\end{align*}
Therefore,
\begin{align*}
(\bm{\tau}_h-\nabla z,\nabla(z-v_h))_{L^2(\Omega)^d}=0.
\end{align*}
Thus,
\begin{align*}
\|\bm{\tau}_h-\nabla v_h\|_{L^2(\Omega)^d}^2
=
|z-v_h|_{H^1(\Omega)}^2
+
\|\nabla z-\bm{\tau}_h\|_{L^2(\Omega)^d}^2.
\end{align*}
Because the two terms on the right-hand side depend separately on $v_h$ and $\bm{\tau}_h$, respectively, the Ritz best-approximation property and Lemma~\ref{lem:best-equilibrated} give
\begin{align*}
&\min_{v_h\in V_h^c}
 \min_{\bm{\tau}_h\in\mathcal Q_h(g_h)}
 \|\bm{\tau}_h-\nabla v_h\|_{L^2(\Omega)^d}^2\\
&\quad=
|z-z_h^c(g_h)|_{H^1(\Omega)}^2
+
\|\nabla z-\bm{\sigma}_h^M(g_h)\|_{L^2(\Omega)^d}^2
=
\|\bm{\sigma}_h^M(g_h)
       -\nabla z_h^c(g_h)\|_{L^2(\Omega)^d}^2.
\end{align*}
Taking square roots proves \eqref{eq:joint-hypercircle}. The uniqueness statements for the two separate minimisation problems show that the minimum is attained uniquely when
\begin{align*}
v_h=z_h^c(g_h),
\quad
\bm{\tau}_h=\bm{\sigma}_h^M(g_h).
\end{align*}
Finally, substituting \eqref{eq:joint-hypercircle} into the definition \eqref{eq:kappaM-def} gives \eqref{eq:kappa-joint-min}.
\end{proof}

\begin{remark}[Relation with the Liu--Oishi constant]\label{rem:LO}
In two dimensions, \eqref{eq:kappa-joint-min} is the lowest-order hypercircle constant used in \citet{LiuOishi2013,TakayasuLiuOishi2014}. Corollary~\ref{cor:joint-hypercircle} identifies the two unique minimisers as $v_h=z_h^c(g_h)$ and $\bm{\tau}_h=\bm{\sigma}_h^M(g_h)$. Thus, $\kappa_{M,h}$ is the CR--Marini realisation of that constant. Equation~\eqref{eq:kappa-joint-min} gives the corresponding simplicial quantity in three dimensions.
\end{remark}

\section{Computable a priori energy and \texorpdfstring{$L^2$}{L2} bounds}\label{sec:apriori}
The proofs use the optimal local Poincar\'e constant $C_{0,h}$ together with $\kappa_{M,h}$, whose finite-dimensional form is given in Section~\ref{sec:matrix}. This gives the mesh-dependent quantity $M_h$. For numerical use, Corollary~\ref{cor:diameter-bound} replaces $C_{0,h}$ by $h/\pi$ and defines $\widehat M_h$ and $\widehat{\mathfrak E}_h$. Section~\ref{sec:numerics} uses only these explicit hatted quantities.

We define
\begin{align}\label{eq:Mh-def}
  M_h:=\sqrt{C_{0,h}^2+\kappa_{M,h}^2}.
\end{align}
For later use, we define also the $f$-dependent quantity
\begin{align}\label{eq:Ehf-def}
  \mathfrak E_h(f)
  :=C_{0,h}\|f-\PiZero f\|_{L^2(\Omega)}
    +\kappa_{M,h}\|\PiZero f\|_{L^2(\Omega)}.
\end{align}
Because $\PiZero$ is the $L^2$-orthogonal projection, the Pythagorean identity and the Cauchy--Schwarz inequality applied to the two orthogonal components give
\begin{align}\label{eq:Ehf-global}
  \mathfrak E_h(f)\le M_h\|f\|_{L^2(\Omega)}.
\end{align}

Let $J:L^2(\Omega)\to V$ denote the Poisson solution operator and $J_h^c:=P_hJ:L^2(\Omega)\to V_h^c$ its conforming Galerkin approximation. We shall use the following operator identity.

\begin{proposition}[Square identity for the Ritz error operator]
\label{prop:operator-square}
We set $E_h^c:=J-J_h^c$. Then, for any $f,g\in L^2(\Omega)$,
\begin{align}\label{eq:operator-square-bilinear}
  (E_h^cf,g)_{L^2(\Omega)}
  =a(E_h^cf,E_h^cg).
\end{align}
Consequently, $E_h^c$ is a bounded positive self-adjoint operator on $L^2(\Omega)$, and
\begin{align}\label{eq:operator-square-norm}
  \|E_h^c\|_{\mathcal L(L^2(\Omega),L^2(\Omega))}
  =
  \|E_h^c\|_{\mathcal L(L^2(\Omega),V)}^2,
\end{align}
where $V$ is equipped with the energy norm $|\cdot|_{H^1(\Omega)}$.
\end{proposition}

\begin{proof}
Because $Jg$ satisfies $a(Jg,v)=(g,v)$ for any $v\in V$, and $E_h^cf$ is energy-orthogonal to $V_h^c$, we have
\begin{align*}
  (E_h^cf,g)
  &=a(E_h^cf,Jg) 
  =a(E_h^cf,Jg-J_h^cg) 
  =a(E_h^cf,E_h^cg).
\end{align*}
This proves \eqref{eq:operator-square-bilinear}. Taking $g=f$ gives
\begin{align*}
  (E_h^cf,f)=|E_h^cf|_{H^1(\Omega)}^2\ge0.
\end{align*}
Furthermore, by the symmetry of $a$,
\begin{align*}
  (E_h^cf,g)
  =a(E_h^cf,E_h^cg)
  =a(E_h^cg,E_h^cf)
  =(f,E_h^cg),
\end{align*}
hence $E_h^c$ is self-adjoint. The Lax--Milgram estimate for $J$, the energy stability of $P_h$, and the Poincar\'e inequality show that $E_h^c$ is bounded on $L^2(\Omega)$. Because $E_h^c$ is a bounded positive self-adjoint operator on
$L^2(\Omega)$,
\begin{align*}
  \|E_h^c\|_{\mathcal L(L^2(\Omega),L^2(\Omega))}
  &=
  \sup_{0\ne f\in L^2(\Omega)}
  \frac{(E_h^cf,f)_{L^2(\Omega)}}
       {\|f\|_{L^2(\Omega)}^2} 
  =
  \sup_{0\ne f\in L^2(\Omega)}
  \frac{|E_h^cf|_{H^1(\Omega)}^2}
       {\|f\|_{L^2(\Omega)}^2} \\
  &=
  \left(
  \sup_{0\ne f\in L^2(\Omega)}
  \frac{|E_h^cf|_{H^1(\Omega)}}
       {\|f\|_{L^2(\Omega)}}
  \right)^2 
  =
  \|E_h^c\|_{\mathcal L(L^2(\Omega),V)}^2.
\end{align*}
Here, the third equality holds because the quotient is non-negative, so taking the square commutes with taking the supremum. This proves \eqref{eq:operator-square-norm}.
\end{proof}

\subsection{The conforming \texorpdfstring{$\Pone$}{P1} approximation}\label{subsec:Lag-L2}
The energy estimate below is the lowest-order form of the computable hypercircle bounds in \citet{LiuOishi2013} and \citet{TakayasuLiuOishi2014}. Here $\kappa_{M,h}$ is realised by the scalar CR--Marini construction.

\begin{theorem}[Energy estimate]\label{thm:energy}
Let $f\in L^2(\Omega)$, and let $u\in V$ and $u_h\in V_h^c$ solve \eqref{eq:poisson} and \eqref{eq:P1}, respectively. Then,
\begin{align}\label{eq:energy-main}
  |u-u_h|_{H^1(\Omega)}
  \le \mathfrak E_h(f)
  \le M_h\|f\|_{L^2(\Omega)}.
\end{align}
No convexity of $\Omega$ and no global $H^2$-regularity are required.
\end{theorem}

\begin{proof}
We set $f_h:=\PiZero f$, and let $\bar u\in V$ solve $-\varDelta\bar u=f_h$.  Let $\bar u_h:=z_h^c(f_h)\in V_h^c$. By the best-approximation property of the Ritz projection,
\begin{align*}
  |u-u_h|_{H^1(\Omega)}
  \le |u-\bar u_h|_{H^1(\Omega)}
  \le |u-\bar u|_{H^1(\Omega)}
     +|\bar u-\bar u_h|_{H^1(\Omega)}.
\end{align*}
Setting $q:=u-\bar u=J(f-f_h)$ yields
\begin{align*}
  a(q,v)=(f-f_h,v)
  \quad\forall v\in V.
\end{align*}
The Cauchy--Schwarz inequality for $a$ gives
\begin{align*}
  \|f-f_h\|_{H^{-1}(\Omega)}
  \le |q|_{H^1(\Omega)}.
\end{align*}
If $q=0$, equality is immediate. Otherwise, taking $v=q$ in the supremum defining the energy-dual norm gives the reverse inequality. Consequently,
\begin{align*}
  |u-\bar u|_{H^1(\Omega)}
  =\|f-f_h\|_{H^{-1}(\Omega)}.
\end{align*}
Lemma~\ref{lem:osc} therefore gives
\begin{align*}
  |u-\bar u|_{H^1(\Omega)}
  \le C_{0,h}\|f-f_h\|_{L^2(\Omega)}.
\end{align*}
Lemma~\ref{lem:hypercircle} gives
\begin{align*}
  |\bar u-\bar u_h|_{H^1(\Omega)}
  \le \kappa_{M,h}\|f_h\|_{L^2(\Omega)}.
\end{align*}
Therefore,
\begin{align*}
  |u-u_h|_{H^1(\Omega)}
  \le C_{0,h}\|f-f_h\|_{L^2(\Omega)}+\kappa_{M,h}\|f_h\|_{L^2(\Omega)}
  =\mathfrak E_h(f).
\end{align*}
As $\PiZero$ is the $L^2$-orthogonal projection,
\begin{align*}
  \|f\|_{L^2(\Omega)}^2=\|f-f_h\|_{L^2(\Omega)}^2+\|f_h\|_{L^2(\Omega)}^2.
\end{align*}
The Cauchy--Schwarz inequality applied to these two orthogonal components gives \eqref{eq:Ehf-global} and completes the proof.
\end{proof}

\begin{theorem}[$L^2$ a priori estimate]\label{thm:L2}
Under the assumptions of Theorem~\ref{thm:energy},
\begin{align}\label{eq:L2-main}
  \|u-u_h\|_{L^2(\Omega)}
  \le M_h|u-u_h|_{H^1(\Omega)}
  \le M_h\mathfrak E_h(f)
  \le M_h^2\|f\|_{L^2(\Omega)}.
\end{align}
Equivalently, because $f=-\varDelta u$,
\begin{align}\label{eq:L2-graph}
  \|u-u_h\|_{L^2(\Omega)}
  \le M_h^2\|\varDelta u\|_{L^2(\Omega)}
\end{align}
for any $u\in V$ such that $\varDelta u\in L^2(\Omega)$, where $u_h=P_hu$
and the Laplacian is understood in the distributional sense.
\end{theorem}

\begin{proof}
Let $e:=u-u_h\in V$, and let $w\in V$ solve
\begin{align}\label{eq:dual-weak}
  a(w,v)=(e,v)_{L^2(\Omega)}
  \quad\forall v\in V.
\end{align}
Because $P_hw\in V_h^c$ satisfies
\begin{align*}
  a(P_hw,v_h)=a(w,v_h)=(e,v_h)_{L^2(\Omega)}
  \quad\forall v_h\in V_h^c,
\end{align*}
it is the conforming Galerkin approximation of $w$. Therefore, Theorem~\ref{thm:energy}, applied to the Poisson problem with right-hand side $e$, gives
\begin{align}\label{eq:dual-energy-computable}
  |w-P_hw|_{H^1(\Omega)}
  \le \mathfrak E_h(e)
  \le M_h\|e\|_{L^2(\Omega)}.
\end{align}
Using the dual equation with $v=e$ and the Galerkin orthogonality $a(e,P_hw)=0$, we obtain
\begin{align*}
  \|e\|_{L^2(\Omega)}^2
  &=(e,e)_{L^2(\Omega)}
   =a(w,e) 
  =a(w-P_hw,e) \\
  &\le |w-P_hw|_{H^1(\Omega)}|e|_{H^1(\Omega)}.
\end{align*}
If $e=0$, the first inequality in \eqref{eq:L2-main} is immediate. Otherwise, using \eqref{eq:dual-energy-computable} and cancelling $\|e\|_{L^2(\Omega)}>0$ gives
\begin{align*}
  \|e\|_{L^2(\Omega)}\le M_h|e|_{H^1(\Omega)}.
\end{align*}
The remaining inequalities follow from Theorem~\ref{thm:energy}.
\end{proof}

\begin{corollary}[Diameter-explicit computable bound]
\label{cor:diameter-bound}
Under the assumptions of Theorem~\ref{thm:energy}, we set
\begin{align}\label{eq:Mhat-def}
  \widehat C_{0,h}:=\frac{h}{\pi},
  \quad
  \widehat M_h
  :=\sqrt{\widehat C_{0,h}^{\,2}+\kappa_{M,h}^2},
\end{align}
and
\begin{align}\label{eq:Ehat-def}
  \widehat{\mathfrak E}_h(f)
  :=
  \widehat C_{0,h}
  \|f-\PiZero f\|_{L^2(\Omega)}
  +
  \kappa_{M,h}
  \|\PiZero f\|_{L^2(\Omega)}.
\end{align}
Then,
\begin{equation}\label{eq:explicit-diameter-bound}
\begin{aligned}
  |u-u_h|_{H^1(\Omega)}
  &\le \widehat{\mathfrak E}_h(f)
  \le \widehat M_h\|f\|_{L^2(\Omega)},\\
  \|u-u_h\|_{L^2(\Omega)}
  &\le
  \widehat M_h\widehat{\mathfrak E}_h(f)
  \le
  \widehat M_h^2\|f\|_{L^2(\Omega)}.
\end{aligned}
\end{equation}
\end{corollary}

\begin{proof}
The Payne--Weinberger estimate \eqref{eq:C0h-diam} gives
\begin{align*}
  C_{0,h}
  \le
  \max_{T\in\mathbb T_h}\frac{h_T}{\pi}
  =
  \frac{h}{\pi}
  =
  \widehat C_{0,h}.
\end{align*}
Therefore,
\begin{align*}
  \mathfrak E_h(f)\le\widehat{\mathfrak E}_h(f),
  \quad
  M_h\le\widehat M_h.
\end{align*}
The $L^2$-orthogonality of $\PiZero$ and the Cauchy--Schwarz inequality applied to the two orthogonal components give
\begin{align*}
  \widehat{\mathfrak E}_h(f)
  \le\widehat M_h\|f\|_{L^2(\Omega)}.
\end{align*}
Theorem~\ref{thm:energy} therefore yields
\begin{align*}
  |u-u_h|_{H^1(\Omega)}
  \le\mathfrak E_h(f)
  \le\widehat{\mathfrak E}_h(f)
  \le\widehat M_h\|f\|_{L^2(\Omega)}.
\end{align*}
Similarly, Theorem~\ref{thm:L2} gives
\begin{align*}
  \|u-u_h\|_{L^2(\Omega)}
  &\le M_h\mathfrak E_h(f) 
  \le\widehat M_h\widehat{\mathfrak E}_h(f) 
  \le\widehat M_h^2\|f\|_{L^2(\Omega)}.
\end{align*}
This proves \eqref{eq:explicit-diameter-bound}.
\end{proof}

\begin{remark}[Role of elliptic regularity]
\label{rem:regularity}
The proof of Theorem~\ref{thm:L2} uses a dual Poisson problem but no $H^2$-regularity estimate for its solution. The same mesh-dependent bound $M_h$ controls the energy-operator norm of the primal and dual Ritz errors. Proposition~\ref{prop:operator-square} shows that the $L^2$-norm of the Ritz error operator is the square of its energy-operator norm, explaining the squared constant in the hypercircle-based estimates of \citet{LiuOishi2013} and \citet{TakayasuLiuOishi2014}.
\end{remark}

\begin{remark}[Piecewise constant data]
\label{rem:P0data}
If $f\in X_h$, the primal oscillation term vanishes and
\begin{equation}\label{eq:P0-special}
\begin{aligned}
  |u-u_h|_{H^1(\Omega)}
  &\le \kappa_{M,h}\|f\|_{L^2(\Omega)},\\
  \|u-u_h\|_{L^2(\Omega)}
  &\le M_h\kappa_{M,h}\|f\|_{L^2(\Omega)}.
\end{aligned}
\end{equation}
The dual factor in the second estimate is $M_h$ rather than $\kappa_{M,h}$, because the right-hand side of the dual problem is $u-u_h$, which need not belong to $X_h$. This is the setting in which the primal CR--Marini estimate appears without an additional restriction on the dual right-hand side.
\end{remark}

\subsection{The Crouzeix--Raviart approximation}\label{subsec:CR-L2}
We now use the same CR space to bound the nonconforming solution in the broken energy norm and in $L^2$. For $f\in L^2(\Omega)$, let $u_h^{\CR}\in V_{h,0}^{\CR}$ solve
\begin{align}\label{eq:CR-primal}
  a_h(u_h^{\CR},v_h)
  :=(\gradH u_h^{\CR},\gradH v_h)
  =(f,v_h)
  \quad\forall v_h\in V_{h,0}^{\CR}.
\end{align}
Let
\begin{align*}
  J_h^{\CR}:L^2(\Omega)\to V_{h,0}^{\CR},
  \quad J_h^{\CR}f:=u_h^{\CR},
\end{align*}
and recall the conforming Galerkin solution operator $J_h^c=P_hJ:L^2(\Omega)\to V_h^c$. Because $V_h^c\subset V_{h,0}^{\CR}$, we may regard $J_h^c$ as an operator from $L^2(\Omega)$ into $V_{h,0}^{\CR}$. We then define the CR--conforming gap operator
\begin{align}\label{eq:Dh-def}
  D_h:=J_h^{\CR}-J_h^c:
  L^2(\Omega)\to V_{h,0}^{\CR}.
\end{align}

\begin{proposition}[Square identity for the CR--conforming gap]
\label{prop:CR-gap-square}
For any $f,g\in L^2(\Omega)$,
\begin{align}\label{eq:CR-gap-square}
  (D_hf,g)=a_h(D_hf,D_hg).
\end{align}
Consequently, $D_h$, viewed as an operator on $L^2(\Omega)$, is positive and self-adjoint. Furthermore,
\begin{align}\label{eq:CR-gap-energy-bound}
  |D_hf|_{H^1(\mathbb T_h)}
  \le
  C_{0,h}\|f-\PiZero f\|_{L^2(\Omega)}
  +\kappa_{M,h}\|\PiZero f\|_{L^2(\Omega)}
  \le M_h\|f\|_{L^2(\Omega)},
\end{align}
and
\begin{align}\label{eq:CR-gap-L2-bound}
  \|D_h\|_{\cL(L^2(\Omega),L^2(\Omega))}
  \le M_h^2 .
\end{align}
\end{proposition}

\begin{proof}
For any $v_h^c\in V_h^c$, the two discrete equations and $V_h^c\subset V_{h,0}^{\CR}$ imply
\begin{align}\label{eq:Dh-orthogonality}
  a_h(D_hf,v_h^c)=0.
\end{align}
Using the CR equation with function $g$,
\begin{align*}
  (D_hf,g)
  =a_h(D_hf,J_h^{\CR}g)
  =a_h(D_hf,D_hg),
\end{align*}
because $J_h^cg\in V_h^c$. This proves \eqref{eq:CR-gap-square}. Taking $g=f$ gives
\begin{align*}
  (D_hf,f)=|D_hf|_{H^1(\mathbb T_h)}^2\ge0.
\end{align*}
The symmetry of $a_h$ in turn gives
\begin{align*}
  (D_hf,g)
  =a_h(D_hf,D_hg)
  =a_h(D_hg,D_hf)
  =(f,D_hg),
\end{align*}
hence $D_h$ is positive and self-adjoint. By well-posedness and finite dimensionality, $J_h^{\CR}$, $J_h^c$, and $D_h$ are bounded as $L^2(\Omega)\to L^2(\Omega)$ operators.

We set $f_h:=\PiZero f$ and $r:=f-f_h$. For the piecewise constant part,
\begin{align*}
  D_hf_h=z_h^{\CR}(f_h)-z_h^c(f_h).
\end{align*}
The exact defect decomposition gives
\begin{align*}
  |D_hf_h|_{H^1(\mathbb T_h)}
  \le
  \|\bm{\sigma}_h^M(f_h)-\nabla z_h^c(f_h)\|_{L^2(\Omega)^d}
  \le \kappa_{M,h}\|f_h\|_{L^2(\Omega)}.
\end{align*}
For the oscillatory part, \eqref{eq:CR-gap-square} gives
\begin{align*}
  |D_hr|_{H^1(\mathbb T_h)}^2=(r,D_hr).
\end{align*}
Because $r$ has zero mean on every element,
\begin{align*}
\begin{aligned}
  (r,D_hr)
  &=
  \sum_{T\in\mathbb{T}_h}(r,D_hr-\Pi_T^0D_hr)_T
  \le C_{0,h}\|r\|_{L^2(\Omega)}\,
  |D_hr|_{H^1(\mathbb T_h)}.
\end{aligned}
\end{align*}
If $D_hr=0$, the desired estimate is immediate. Otherwise, division by $|D_hr|_{H^1(\mathbb T_h)}>0$ gives
\begin{align*}
  |D_hr|_{H^1(\mathbb T_h)}\le C_{0,h}\|r\|_{L^2(\Omega)}.
\end{align*}
By linearity and the triangle inequality,
\begin{align*}
  |D_hf|_{H^1(\mathbb T_h)}
  \le C_{0,h}\|r\|_{L^2(\Omega)}+\kappa_{M,h}\|f_h\|_{L^2(\Omega)}.
\end{align*}
Because $\PiZero$ is the $L^2$-orthogonal projection, the Cauchy--Schwarz inequality applied to the two orthogonal components yields \eqref{eq:CR-gap-energy-bound}.

Finally, positivity and self-adjointness give
\begin{align*}
  \|D_h\|_{\cL(L^2(\Omega),L^2(\Omega))}
  =
  \sup_{0\ne f\in L^2(\Omega)}
  \frac{(D_hf,f)}{\|f\|_{L^2(\Omega)}^2}
  =
  \sup_{0\ne f\in L^2(\Omega)}
  \frac{|D_hf|_{H^1(\mathbb T_h)}^2}{\|f\|_{L^2(\Omega)}^2}
  \le M_h^2.
\end{align*}
\end{proof}

\begin{theorem}[Energy and $L^2$ a priori estimates for the CR solution]
\label{thm:CR-L2}
Let $f\in L^2(\Omega)$, let $u=Jf\in V$ solve \eqref{eq:poisson}, and let $u_h^{\CR}=J_h^{\CR}f$ solve \eqref{eq:CR-primal}. Then,
\begin{align}\label{eq:CR-energy-main}
  |u-u_h^{\CR}|_{H^1(\mathbb T_h)}
  \le 2\mathfrak E_h(f)
  \le 2M_h\|f\|_{L^2(\Omega)},
\end{align}
and
\begin{align}\label{eq:CR-L2-main}
  \|u-u_h^{\CR}\|_{L^2(\Omega)}
  \le M_h^2\|f\|_{L^2(\Omega)}.
\end{align}
These bounds hold on any non-degenerate conforming simplicial mesh and require neither convexity of $\Omega$ nor global $H^2$-regularity.
\end{theorem}

\begin{proof}
Let $u_h^c:=J_h^cf$. By the definition of $D_h$,
\begin{align*}
  D_hf=u_h^{\CR}-u_h^c.
\end{align*}
Because $u-u_h^c\in V$, its broken and global energy seminorms coincide. The triangle inequality, Theorem~\ref{thm:energy}, and \eqref{eq:CR-gap-energy-bound} therefore give
\begin{align*}
  |u-u_h^{\CR}|_{H^1(\mathbb T_h)}
  &\le |u-u_h^c|_{H^1(\Omega)}
       +|D_hf|_{H^1(\mathbb T_h)} \\
  &\le 2\mathfrak E_h(f)
  \le 2M_h\|f\|_{L^2(\Omega)}.
\end{align*}
This proves \eqref{eq:CR-energy-main}.

Recall $E_h^c=J-J_h^c$, and we set
\begin{align*}
  E_h^{\CR}:=J-J_h^{\CR}.
\end{align*}
Proposition~\ref{prop:operator-square} shows that $E_h^c$ is positive and self-adjoint and that
\begin{align*}
  \|E_h^c\|_{\cL(L^2(\Omega),L^2(\Omega))}
  =
  \|E_h^c\|_{\cL(L^2(\Omega),V)}^2.
\end{align*}
Furthermore, Theorem~\ref{thm:energy} and \eqref{eq:Ehf-global} give
\begin{align*}
  \|E_h^c\|_{\cL(L^2(\Omega),V)}
  \le M_h.
\end{align*}
Therefore,
\begin{align*}
  \|E_h^c\|_{\cL(L^2(\Omega),L^2(\Omega))}
  \le M_h^2.
\end{align*}
By Proposition~\ref{prop:CR-gap-square}, $D_h$ is also positive and self-adjoint and
\begin{align*}
  \|D_h\|_{\cL(L^2(\Omega),L^2(\Omega))}\le M_h^2.
\end{align*}
The identity
\begin{align*}
  E_h^{\CR}=E_h^c-D_h,
\end{align*}
shows that $E_h^{\CR}$ is bounded and self-adjoint. For any $g\in L^2(\Omega)$, positivity and the preceding operator-norm bounds give
\begin{align*}
  0
  &\le (E_h^cg,g)
  \le M_h^2\|g\|_{L^2(\Omega)}^2,\\
  0
  &\le (D_hg,g)
  \le M_h^2\|g\|_{L^2(\Omega)}^2.
\end{align*}
Indeed, for either $A=E_h^c$ or $A=D_h$, positivity and the
Cauchy--Schwarz inequality give
\begin{align*}
  0\le (Ag,g)
  \le \|A\|_{\cL(L^2(\Omega),L^2(\Omega))}
       \|g\|_{L^2(\Omega)}^2
  \le M_h^2\|g\|_{L^2(\Omega)}^2.
\end{align*}
Consequently,
\begin{align*}
\begin{aligned}
 |(E_h^{\CR}g,g)|
 &=
 |(E_h^cg,g)-(D_hg,g)| 
 \le
 \max\{(E_h^cg,g),(D_hg,g)\} 
 \le M_h^2\|g\|_{L^2(\Omega)}^2.
\end{aligned}
\end{align*}
The norm of a bounded self-adjoint operator equals the supremum of the absolute Rayleigh quotient. Therefore,
\begin{align*}
  \|E_h^{\CR}\|_{\cL(L^2(\Omega),L^2(\Omega))}\le M_h^2.
\end{align*}
Finally, because $u-u_h^{\CR}=E_h^{\CR}f$,
\begin{align*}
  \|u-u_h^{\CR}\|_{L^2(\Omega)}
  &\le
  \|E_h^{\CR}\|_{\cL(L^2(\Omega),L^2(\Omega))}
  \|f\|_{L^2(\Omega)} 
  \le M_h^2\|f\|_{L^2(\Omega)}.
\end{align*}
This proves \eqref{eq:CR-L2-main}.
\end{proof}

\begin{corollary}[Refined CR $L^2$ bound]
\label{cor:CR-refined}
Under the assumptions of Theorem~\ref{thm:CR-L2},
\begin{align}\label{eq:CR-refined}
  \|u-u_h^{\CR}\|_{L^2(\Omega)}
  \le
  \min\left\{
    M_h^2\|f\|_{L^2(\Omega)},
    \;2M_h\mathfrak E_h(f)
  \right\}.
\end{align}
\end{corollary}

\begin{proof}
The first bound is Theorem~\ref{thm:CR-L2}. For the second one, the conforming estimate gives
\begin{align*}
  \|E_h^cf\|_{L^2(\Omega)}
  \le M_h|E_h^cf|_{H^1(\Omega)}
  \le M_h\mathfrak E_h(f).
\end{align*}
Because $D_h$ is positive and self-adjoint and $\|D_h\|_{\cL(L^2,L^2)}\le M_h^2$, its spectrum is contained in $[0,\|D_h\|]$. Therefore, $\lambda^2\le\|D_h\|\lambda$ on the spectrum, and thus
\begin{align*}
  (D_h^2f,f)
  \le \|D_h\|(D_hf,f)
  \quad\forall f\in L^2(\Omega).
\end{align*}
This inequality and \eqref{eq:CR-gap-square} yield
\begin{align*}
  \|D_hf\|_{L^2(\Omega)}^2
  &=(D_h^2f,f)
  \le M_h^2(D_hf,f)
  =M_h^2|D_hf|_{H^1(\mathbb T_h)}^2
  \le M_h^2\mathfrak E_h(f)^2.
\end{align*}
Because $E_h^{\CR}=E_h^c-D_h$, the triangle inequality gives
\begin{align*}
  \|E_h^{\CR}f\|_{L^2(\Omega)}
  \le 2M_h\mathfrak E_h(f).
\end{align*}
Combining the two valid bounds proves \eqref{eq:CR-refined}.
\end{proof}

\begin{corollary}[Diameter-explicit CR bounds]\label{cor:CR-diameter}
Under the assumptions of Theorem~\ref{thm:CR-L2}, and with $\widehat{\mathfrak E}_h$ and $\widehat M_h$ from Corollary~\ref{cor:diameter-bound},
\begin{align}\label{eq:CR-explicit-energy}
  |u-u_h^{\CR}|_{H^1(\mathbb T_h)}
  \le 2\widehat{\mathfrak E}_h(f)
  \le 2\widehat M_h\|f\|_{L^2(\Omega)},
\end{align}
and
\begin{align}\label{eq:CR-explicit}
  \|u-u_h^{\CR}\|_{L^2(\Omega)}
  \le
  \min\left\{
    \widehat M_h^2\|f\|_{L^2(\Omega)},
    \;2\widehat M_h\widehat{\mathfrak E}_h(f)
  \right\}.
\end{align}
\end{corollary}

\begin{proof}
As in the proof of Corollary~\ref{cor:diameter-bound}, $C_{0,h}\le\widehat C_{0,h}$ and the $L^2$-orthogonality of $\PiZero$ give
\begin{align*}
  \mathfrak E_h(f)
  \le \widehat{\mathfrak E}_h(f)
  \le \widehat M_h\|f\|_{L^2(\Omega)},
  \quad
  M_h\le\widehat M_h.
\end{align*}
The three estimates now follow from Theorem~\ref{thm:CR-L2} and Corollary~\ref{cor:CR-refined}:
\begin{align*}
  |u-u_h^{\CR}|_{H^1(\mathbb T_h)}
  &\le 2\mathfrak E_h(f)
  \le 2\widehat{\mathfrak E}_h(f)
  \le 2\widehat M_h\|f\|_{L^2(\Omega)},\\
  \|u-u_h^{\CR}\|_{L^2(\Omega)}
  &\le M_h^2\|f\|_{L^2(\Omega)}
  \le \widehat M_h^2\|f\|_{L^2(\Omega)},\\
  \|u-u_h^{\CR}\|_{L^2(\Omega)}
  &\le 2M_h\mathfrak E_h(f)
  \le 2\widehat M_h\widehat{\mathfrak E}_h(f).
\end{align*}
Taking the minimum of the two upper bounds proves \eqref{eq:CR-explicit}.
\end{proof}

\begin{remark}[Structure of the CR argument]\label{rem:CR-proof}
The broken-energy estimate in Theorem~\ref{thm:CR-L2} is obtained through the conforming solution $u_h^c$. The $L^2$ estimate is not derived from a broken-energy duality inequality; it follows from $V_h^c\subset V_{h,0}^{\CR}$ and positivity of the gap operator $D_h=J_h^{\CR}-J_h^c$.
\end{remark}

\section{Finite-dimensional form of the Marini constant}\label{sec:matrix}
The supremum in \eqref{eq:kappaM-def} can be written as a finite-dimensional generalised eigenvalue problem. Only scalar conforming and CR matrices and explicit element geometry are needed; an $\RT^0$ basis is not assembled.

Let $\{\psi_j\}_{j=1}^N$ be a basis of $X_h$. For $\mathbf g=(g_1,\ldots,g_N)^{\top} \in\R^N$, we set
\begin{align*}
  g_h:=\sum_{j=1}^N g_j\psi_j.
\end{align*}
The mass matrix $M_0\in\R^{N\times N}$ of $X_h$ is defined as
\begin{align*}
  (M_0)_{ij}:=(\psi_j,\psi_i)_{L^2(\Omega)}.
\end{align*}
It is symmetric positive definite and satisfies
\begin{align}\label{eq:M0-quadratic}
  \|g_h\|_{L^2(\Omega)}^2=\mathbf g^{\top} M_0\mathbf g.
\end{align}

We choose bases $\{\phi_i^c\}_{i=1}^{n_c}$ of $V_h^c$ and $\{\phi_i^{\CR}\}_{i=1}^{n_{\CR}}$ of $V_{h,0}^{\CR}$. We use the row-test-function convention
\begin{align*}
  (A_c)_{ik}
  &:=a(\phi_k^c,\phi_i^c),
  \quad
  (B_c)_{ij}
  :=(\psi_j,\phi_i^c)_{L^2(\Omega)},\\
  (A_{\CR})_{ik}
  &:=a_h(\phi_k^{\CR},\phi_i^{\CR}),
  \quad
  (B_{\CR})_{ij}
  :=(\psi_j,\phi_i^{\CR})_{L^2(\Omega)}.
\end{align*}
The matrices $A_c$ and $A_{\CR}$ are symmetric positive definite. If $\mathbf z_c$ and $\mathbf z_{\CR}$ are the coefficient vectors of $z_h^c(g_h)$ and $z_h^{\CR}(g_h)$ in the chosen bases, respectively, then
\begin{align}\label{eq:scalar-solution-vectors}
  \mathbf z_c=A_c^{-1}B_c\mathbf g,
  \quad
  \mathbf z_{\CR}=A_{\CR}^{-1}B_{\CR}\mathbf g.
\end{align}
Consequently,
\begin{align*}
  |z_h^c(g_h)|_{H^1(\Omega)}^2
  &=\mathbf g^{\top} B_c^{\! {\top}}A_c^{-1}B_c\mathbf g,\\
  |z_h^{\CR}(g_h)|_{H^1(\mathbb T_h)}^2
  &=\mathbf g^{\top} B_{\CR}^{\! {\top}}A_{\CR}^{-1}B_{\CR}\mathbf g.
\end{align*}
We define
\begin{align}\label{eq:Ggap}
  G_{\rm gap}
  :=B_{\CR}^{\! {\top}}A_{\CR}^{-1}B_{\CR}
    -B_c^{\! {\top}}A_c^{-1}B_c.
\end{align}
The energy-gap identity \eqref{eq:energy-gap} gives, for any $\mathbf g\in\R^N$,
\begin{align}\label{eq:Ggap-quadratic}
  \mathbf g^{\top} G_{\rm gap}\mathbf g
  =|z_h^{\CR}(g_h)-z_h^c(g_h)|_{H^1(\mathbb T_h)}^2
  \ge0.
\end{align}
Therefore, $G_{\rm gap}$ is symmetric positive semidefinite. This matrix is the finite-dimensional form of the CR--conforming energy gap.

The second term in Proposition~\ref{prop:defect-decomp} is represented by $D_M\in\R^{N\times N}$, defined by
\begin{align}\label{eq:DM}
  (D_M)_{ij}
  :=\frac1{d^2}\sum_{T\in\mathbb{T}_h}\mu_T
       (\psi_j|_T)(\psi_i|_T).
\end{align}
Indeed, writing $g_T:=g_h|_T$, we have
\begin{align}\label{eq:DM-quadratic}
  \mathbf g^{\top} D_M\mathbf g
  =\frac1{d^2}\sum_{T\in\mathbb T_h}\mu_T|g_T|^2.
\end{align}
Thus, $D_M$ is symmetric positive definite and contains only the elementwise second moments $\mu_T$. For the element-characteristic basis, we index the basis functions by the mesh elements and we set
\begin{align*}
  \psi_T:=\mathbf 1_T,
  \quad T\in\mathbb T_h.
\end{align*}
Accordingly, the rows and columns of $M_0$ and $D_M$ are indexed by $S,T\in\mathbb T_h$. Both matrices are diagonal. For an element $T$ with vertices $\bm x_1,\ldots,\bm x_{d+1}$,
\begin{align*}
  (M_0)_{ST}
  &=\delta_{ST}|T|_d,\\
  (D_M)_{ST}
  &=\delta_{ST}
    \frac{|T|_d}{d^2(d+1)^2(d+2)}
    \sum_{1\le i<j\le d+1}|\bm x_i-\bm x_j|^2.
\end{align*}
Here, $\delta_{ST}$ denotes the Kronecker delta. In particular, $(M_0)_{TT}=|T|_d$ is the mass-matrix entry associated with the basis function $\mathbf 1_T$, and
\begin{align}\label{eq:DM-diagonal-ratio}
  \frac{(D_M)_{TT}}{(M_0)_{TT}}
  =\frac{1}{d^2(d+1)^2(d+2)}
    \sum_{1\le i<j\le d+1}|\bm x_i-\bm x_j|^2.
\end{align}

To identify the matrix of the complete Marini defect, we define the symmetric matrix $G_M\in\R^{N\times N}$ as
\begin{align}\label{eq:GM-def}
  (G_M)_{ij}
  :=\left(
    \bm{\sigma}_h^M(\psi_j)-\nabla z_h^c(\psi_j),
    \bm{\sigma}_h^M(\psi_i)-\nabla z_h^c(\psi_i)
    \right)_{L^2(\Omega)^d}.
\end{align}
The linearity of the conforming and CR solution operators, and hence of the Marini reconstruction, implies
\begin{align}\label{eq:GM-quadratic}
  \mathbf g^{\top} G_M\mathbf g
  =\|\bm{\sigma}_h^M(g_h)-\nabla z_h^c(g_h)\|_{L^2(\Omega)^d}^2.
\end{align}
Combining \eqref{eq:Ggap-quadratic}, \eqref{eq:DM-quadratic}, and Proposition~\ref{prop:defect-decomp}, we obtain equality of the corresponding quadratic forms for every $\mathbf g\in\R^N$. Since all three matrices are symmetric, this proves the exact matrix identity
\begin{align}\label{eq:GM-scalar}
  G_M=G_{\rm gap}+D_M.
\end{align}

Equations \eqref{eq:kappaM-def}, \eqref{eq:M0-quadratic}, and \eqref{eq:GM-scalar} therefore yield
\begin{align}
  \kappa_{M,h}^2
  &=\max_{0\ne\mathbf g\in\R^N}
    \frac{\mathbf g^{\top} (G_{\rm gap}+D_M)\mathbf g}
         {\mathbf g^{\top} M_0\mathbf g} \notag\\
  &=\lambda_{\max}(G_{\rm gap}+D_M,M_0).
  \label{eq:kappa-eig}
\end{align}
Here, $\lambda_{\max}(A,M)$ denotes the largest generalised eigenvalue of $A\mathbf v=\lambda M\mathbf v$. Although these matrices depend on the chosen basis of $X_h$, the generalised eigenvalue, and hence $\kappa_{M,h}$, does not.

One direct realisation is to solve
\begin{align*}
  A_cZ_c=B_c,
  \quad
  A_{\CR}Z_{\CR}=B_{\CR},
\end{align*}
with multiple right-hand sides and then form $B_c^{\! \top}Z_c$ and $B_{\CR}^{\! \top}Z_{\CR}$. For large $\dim X_h$ we instead use a matrix-free largest-eigenvalue iteration. The action of $G_{\rm gap}$ on $\mathbf g$ requires one conforming and one CR sparse solve, with reusable factorisations, while $D_M$ is diagonal in the element-characteristic basis. No dense gap matrix, mixed saddle-point system, or $\RT^0$ basis is formed.

\begin{remark}[Rigorous numerical enclosure]\label{rem:verified-matrix}
A certified numerical upper bound for $\kappa_{M,h}$ would require outward-rounded enclosures for the matrix assemblies, linear solves, and largest generalised eigenvalue in \eqref{eq:kappa-eig}. For a termwise upper enclosure of $G_{\rm gap}$ in the Loewner order, the CR contribution must be bounded from above and the subtracted conforming contribution from below. The computations in Section~\ref{sec:numerics} use standard floating-point arithmetic and do not provide such enclosures.
\end{remark}

\section{Anisotropic geometry of the computable constant}\label{sec:anisotropic}
The single-mesh estimates of Theorems~\ref{thm:energy}, \ref{thm:L2}, and \ref{thm:CR-L2} require neither shape regularity nor a maximum-angle condition. A rate for $\kappa_{M,h}$ or $M_h$, however, is a statement about a mesh family and needs further geometric information.

The matrix identity \eqref{eq:GM-scalar} separates the two contributions to $\kappa_{M,h}$. We define
\begin{align}\label{eq:gamma-gap-geo}
  \gamma_{\rm gap,h}^2
  :=\lambda_{\max}(G_{\rm gap},M_0),
  \quad
  \gamma_{\rm geo,h}^2
  :=\lambda_{\max}(D_M,M_0).
\end{align}
From \eqref{eq:M0-quadratic}, \eqref{eq:Ggap-quadratic}, and \eqref{eq:DM-quadratic}, these quantities have the basis-independent variational characterisations
\begin{align*}
  \gamma_{\rm gap,h}^2
  &=\max_{0\ne g_h\in X_h}
    \frac{|z_h^{\CR}(g_h)-z_h^c(g_h)|_{H^1(\mathbb T_h)}^2}
         {\|g_h\|_{L^2(\Omega)}^2},\\
  \gamma_{\rm geo,h}^2
  &=\max_{0\ne g_h\in X_h}
    \frac{d^{-2}\sum_{T\in\mathbb T_h}\mu_T|g_T|^2}
         {\|g_h\|_{L^2(\Omega)}^2},
  \quad g_T:=g_h|_T.
\end{align*}
Thus, $\gamma_{\rm gap,h}$ is the norm of $D_h|_{X_h}:X_h\to V_{h,0}^{\CR}$, with the $L^2$ norm on $X_h$ and the broken $H^1$ norm on the target space, whereas $\gamma_{\rm geo,h}$ contains only the explicit Marini second moment.

\begin{proposition}[Spectral separation of the Marini constant]
\label{prop:anisotropic-split}
The three computable quantities in \eqref{eq:kappa-eig} and \eqref{eq:gamma-gap-geo} satisfy
\begin{align}\label{eq:anisotropic-spectral-sandwich}
  \max\{\gamma_{\rm gap,h}^2,\gamma_{\rm geo,h}^2\}
  \le \kappa_{M,h}^2
  \le \gamma_{\rm gap,h}^2+\gamma_{\rm geo,h}^2.
\end{align}
For the element-characteristic basis of $X_h$, let $\bm x_{T,1},\ldots,\bm x_{T,d+1}$ denote the vertices of $T$. Then,
\begin{align}\label{eq:gamma-geo-exact}
  \frac{h^2}{d^2(d+1)^2(d+2)}
  \le \gamma_{\rm geo,h}^2
  &=\max_{T\in\mathbb{T}_h}
    \frac{1}{d^2(d+1)^2(d+2)}
    \sum_{1\le i<j\le d+1}|\bm x_{T,i}-\bm x_{T,j}|^2 \notag\\
  &\le \frac{h^2}{2d(d+1)(d+2)}.
\end{align}
Equivalently,
\begin{align}\label{eq:gamma-geo-two-sided}
  \frac{h}{d(d+1)\sqrt{d+2}}
  \le \gamma_{\rm geo,h}
  \le \frac{h}{\sqrt{2d(d+1)(d+2)}}.
\end{align}
In particular, these bounds are $h/12\le\gamma_{\rm geo,h}\le h/\sqrt{48}$ for $d=2$ and $h/\sqrt{720}\le\gamma_{\rm geo,h}\le h/\sqrt{120}$ for $d=3$. The exact geometric term records the edge lengths, but neither its formula nor these bounds contain an inverse altitude or a minimum-angle or aspect-ratio factor.
\end{proposition}

\begin{proof}
For $0\ne\mathbf g\in\R^N$, we define
\begin{align*}
  R_{\rm gap}(\mathbf g)
  &:=
  \frac{\mathbf g^{\top}G_{\rm gap}\mathbf g}
       {\mathbf g^{\top}M_0\mathbf g}, \quad
  R_{\rm geo}(\mathbf g)
  :=
  \frac{\mathbf g^{\top}D_M\mathbf g}
       {\mathbf g^{\top}M_0\mathbf g}.
\end{align*}
By the definitions of the three largest generalised eigenvalues,
\begin{align*}
  \kappa_{M,h}^2
  &=
  \max_{0\ne\mathbf g\in\R^N}
  \left (R_{\rm gap}(\mathbf g)+R_{\rm geo}(\mathbf g)\right),\\
  \gamma_{\rm gap,h}^2
  &=\max_{0\ne\mathbf g\in\R^N}R_{\rm gap}(\mathbf g),\quad
  \gamma_{\rm geo,h}^2
  =\max_{0\ne\mathbf g\in\R^N}R_{\rm geo}(\mathbf g).
\end{align*}
Because $G_{\rm gap}$ and $D_M$ are positive semidefinite, both Rayleigh quotients are non-negative. It follows that
\begin{align*}
  \kappa_{M,h}^2
  \ge\gamma_{\rm gap,h}^2,
  \quad
  \kappa_{M,h}^2
  \ge\gamma_{\rm geo,h}^2.
\end{align*}
Furthermore, for any $0\ne\mathbf g\in\R^N$,
\begin{align*}
  R_{\rm gap}(\mathbf g)+R_{\rm geo}(\mathbf g)
  \le
  \gamma_{\rm gap,h}^2+\gamma_{\rm geo,h}^2.
\end{align*}
Taking the maximum over $\mathbf g$ proves \eqref{eq:anisotropic-spectral-sandwich}.

For the element-characteristic basis, we write
\begin{align*}
  m_T:=(M_0)_{TT},
  \quad
  d_T:=(D_M)_{TT}.
\end{align*}
Since $M_0$ and $D_M$ are diagonal, for $\mathbf g=(g_T)_{T\in\mathbb T_h}\ne0$ we have
\begin{align*}
  \frac{\mathbf g^{\top}D_M\mathbf g}
       {\mathbf g^{\top}M_0\mathbf g}
  =
  \frac{\sum_{T\in\mathbb T_h}d_T|g_T|^2}
       {\sum_{T\in\mathbb T_h}m_T|g_T|^2}.
\end{align*}
This is a weighted average of the ratios $d_T/m_T$. Therefore,
\begin{align*}
  \gamma_{\rm geo,h}^2
  =
  \max_{T\in\mathbb T_h}\frac{d_T}{m_T}.
\end{align*}
Using \eqref{eq:DM-diagonal-ratio} gives
\begin{align*}
  \gamma_{\rm geo,h}^2
  =
  \max_{T\in\mathbb T_h}
  \frac{1}{d^2(d+1)^2(d+2)}
  \sum_{1\le i<j\le d+1}
  |\bm x_{T,i}-\bm x_{T,j}|^2.
\end{align*}
We choose $T_*\in\mathbb T_h$ such that $h_{T_*}=h$. Because the diameter of a simplex is the length of one of its edges,
\begin{align*}
  \sum_{1\le i<j\le d+1}
  |\bm x_{T_*,i}-\bm x_{T_*,j}|^2
  \ge h^2.
\end{align*}
This proves the lower bound in \eqref{eq:gamma-geo-exact}.

A $d$-simplex has $d(d+1)/2$ edges, and any edge of $T$ has length at most $h_T\le h$. Therefore,
\begin{align*}
  \sum_{1\le i<j\le d+1}
  |\bm x_{T,i}-\bm x_{T,j}|^2
  \le
  \frac{d(d+1)}{2}h^2.
\end{align*}
Substitution into the exact formula for $\gamma_{\rm geo,h}^2$ proves the upper bound in \eqref{eq:gamma-geo-exact}. Taking square roots yields \eqref{eq:gamma-geo-two-sided}.
\end{proof}

\begin{corollary}[Exact addition when the geometric coefficient is elementwise constant]
\label{cor:exact-gap-geo-addition}
We define
\begin{align*}
  c_T
  :=\frac{1}{d^2(d+1)^2(d+2)}
    \sum_{1\le i<j\le d+1}
    |\bm x_{T,i}-\bm x_{T,j}|^2.
\end{align*}
If there exists a number $c_h>0$, independent of $T\in\mathbb T_h$, such that $c_T=c_h$, then $D_M=c_hM_0$ and
\begin{align}\label{eq:exact-gap-geo-addition}
  \kappa_{M,h}^2
  =
  \gamma_{\rm gap,h}^2+\gamma_{\rm geo,h}^2.
\end{align}
In particular, this holds on any mesh consisting of congruent simplices. The condition is weaker than congruence and explains the exact addition observed for several mesh families in Section~\ref{sec:numerics}.
\end{corollary}

\begin{proof}
For the element-characteristic basis, \eqref{eq:DM-diagonal-ratio} gives
\begin{align*}
  (D_M)_{TT}=c_T(M_0)_{TT}.
\end{align*}
Therefore, the hypothesis $c_T=c_h$ for any $T\in\mathbb T_h$ implies
\begin{align*}
  D_M=c_hM_0,
  \quad
  \gamma_{\rm geo,h}^2=c_h.
\end{align*}
Thus,
\begin{align*}
  \kappa_{M,h}^2
  &=
  \lambda_{\max}(G_{\rm gap}+c_hM_0,M_0)
  =
  \lambda_{\max}(G_{\rm gap},M_0)+c_h
  =
  \gamma_{\rm gap,h}^2+\gamma_{\rm geo,h}^2.
\end{align*}
\end{proof}

\begin{corollary}[Gap criterion for first-order decay]\label{cor:gap-first-order}
Combining \eqref{eq:anisotropic-spectral-sandwich} and \eqref{eq:gamma-geo-exact} gives
\begin{align}\label{eq:gap-first-order-sandwich}
  \max\left\{
    \gamma_{\rm gap,h},
    \frac{h}{d(d+1)\sqrt{d+2}}
  \right\}
  \le \kappa_{M,h}
  \le
  \left(
    \gamma_{\rm gap,h}^2
    +\frac{h^2}{2d(d+1)(d+2)}
  \right)^{1/2}.
\end{align}
Therefore, along any mesh family with $h\to0$,
\begin{align*}
  \kappa_{M,h}=O(h)
  \quad\Longleftrightarrow\quad
  \gamma_{\rm gap,h}=O(h),
\end{align*}
and $\kappa_{M,h}$ is not $o(h)$.
\end{corollary}

\begin{proof}
The first bound in \eqref{eq:gap-first-order-sandwich} follows from
\begin{align*}
  \kappa_{M,h}
  \ge
  \max\{\gamma_{\rm gap,h},\gamma_{\rm geo,h}\}
\end{align*}
and the lower bound for $\gamma_{\rm geo,h}$ in \eqref{eq:gamma-geo-two-sided}. The upper bound follows similarly from
\begin{align*}
  \kappa_{M,h}^2
  \le
  \gamma_{\rm gap,h}^2+\gamma_{\rm geo,h}^2
\end{align*}
and the upper bound for $\gamma_{\rm geo,h}$. If $\kappa_{M,h}=O(h)$, then $\gamma_{\rm gap,h}\le\kappa_{M,h}$ implies $\gamma_{\rm gap,h}=O(h)$. Conversely, if $\gamma_{\rm gap,h}=O(h)$, the upper bound in \eqref{eq:gap-first-order-sandwich} gives $\kappa_{M,h}=O(h)$. Finally, the lower bound gives
\begin{align*}
  \frac{\kappa_{M,h}}{h}
  \ge
  \frac{1}{d(d+1)\sqrt{d+2}},
\end{align*}
hence $\kappa_{M,h}$ cannot be $o(h)$.
\end{proof}

Consequently, the majorant coefficient $M_h^2$ cannot be $o(h^2)$, although the error for a particular right-hand side may converge faster.

The bounds above show that $\gamma_{\rm geo,h}$ is comparable to $h$, with constants depending only on the dimension. Thus, if $\kappa_{M,h}$ decays more slowly than $h$ along a mesh family, the loss must come from $\gamma_{\rm gap,h}$. We estimate this gap below by directional interpolation. On a uniformly semi-regular family, that is, when $H_T/h_T$ remains bounded, the resulting estimates are uniform with respect to the aspect ratio, provided the required regularity or graded-mesh approximation estimates are available.

Two types of families will be used below. Under similarity refinement, a fixed element shape is scaled uniformly in all of its directional lengths. Under grading, the local scale changes with position; the grading may also be anisotropic, so the aspect ratio can diverge while $H_T/h_T$ remains bounded.

The directional comparison in Subsection~\ref{subsec:directional-anisotropic-comparison} covers both similarity-refined and semi-regular graded meshes. Subsection~\ref{subsec:graded-theory} deals specifically with corner grading, including a tensor-product family with unbounded aspect ratio.

\subsection{Directional anisotropic comparison} \label{subsec:directional-anisotropic-comparison}
For rates we return to Corollary~\ref{cor:joint-hypercircle} and compare the Marini pair with an arbitrary conforming approximation and an arbitrary equilibrated flux.

Fix $g_h\in X_h$, and set $z:=Jg_h$. For any $v_h\in V_h^c$ and $\bm{\tau}_h\in\mathcal Q_h(g_h)$, Corollary~\ref{cor:joint-hypercircle} gives
\begin{align}\label{eq:general-anisotropic-comparison}
  \|\bm{\sigma}_h^M(g_h)-\nabla z_h^c(g_h)\|_{L^2(\Omega)^d}^2
  &\le \|\bm{\tau}_h-\nabla v_h\|_{L^2(\Omega)^d}^2 \notag\\
  &=|z-v_h|_{H^1(\Omega)}^2
    +\|\nabla z-\bm{\tau}_h\|_{L^2(\Omega)^d}^2.
\end{align}
Whenever $z$ lies in the domain of the conforming nodal interpolant and $\nabla z$ lies in the domain of the canonical lowest-order RT interpolant, the commuting property and $\varDelta z=-g_h$ give
\begin{align*}
  \divop(I_h^{\RT}\nabla z)
  =\PiZero(\varDelta z)
  =-g_h.
\end{align*}
Since $z\in H_0^1(\Omega)$ and $\varDelta z=-g_h\in L^2(\Omega)$, $\nabla z\in H(\divop;\Omega)$ and its normal traces are compatible across interior faces. This uses only the graph-space regularity of $z$ and does not imply $z\in H^2(\Omega)$. The RT domain condition means that the face-flux moments defining the canonical interpolant are meaningful; elementwise $W^{1,1}$ regularity of $\nabla z$ is sufficient when these moments are written as ordinary face integrals. Hence,
\begin{align*}
  \bm{\tau}_h:=I_h^{\RT}\nabla z\in\mathcal Q_h(g_h),
  \quad
  v_h:=I_h^c z\in V_h^c.
\end{align*}
Thus, \eqref{eq:general-anisotropic-comparison}, applied to this particular admissible pair, gives
\begin{align}\label{eq:smooth-anisotropic-comparison}
  \|\bm{\sigma}_h^M(g_h)-\nabla z_h^c(g_h)\|_{L^2(\Omega)^d}^2
  &\le \|\bm{\tau}_h-\nabla v_h\|_{L^2(\Omega)^d}^2 \notag\\
  &=|z-I_h^c z|_{H^1(\Omega)}^2
    +\|\nabla z-I_h^{\RT}\nabla z\|_{L^2(\Omega)^d}^2.
\end{align}
This interpolated pair is used only for the rate comparison. It is not part of the proofs of Theorems~\ref{thm:energy}, \ref{thm:L2}, or \ref{thm:CR-L2}, and it does not impose a global $H^2(\Omega)$ assumption on the single-mesh estimates.

We make the anisotropic content of this comparison explicit in two dimensions. For each $T\in\mathbb T_h$, number its vertices $\bm p_{T,1},\bm p_{T,2},\bm p_{T,3}$ so that the edge joining $\bm p_{T,2}$ and $\bm p_{T,3}$ is a longest edge and $|\bm p_{T,1}-\bm p_{T,3}|\le |\bm p_{T,1}-\bm p_{T,2}|$, as in the standard-position construction of \citet{Ishizaka2025Interpolation}. We set
\begin{align}\label{eq:anisotropic-element-geometry}
  h_{T,1}&:=|\bm p_{T,2}-\bm p_{T,1}|,
  \quad \bm r_{T,1}:=\frac{\bm p_{T,2}-\bm p_{T,1}}{h_{T,1}}, \notag\\
  h_{T,2}&:=|\bm p_{T,3}-\bm p_{T,1}|,
  \quad \bm r_{T,2}:=\frac{\bm p_{T,3}-\bm p_{T,1}}{h_{T,2}},
  \quad H_T:=\frac{h_{T,1}h_{T,2}}{|T|_d}h_T.
\end{align}
Here, $h_T$ is the diameter of $T$. For a function $v$ for which the following local directional quantities are finite, and for $g\in L^2(\Omega)$, we define
\begin{align}\label{eq:anisotropic-directional-functionals}
  \mathcal A_{L,h}(v)^2
  &:=\sum_{T\in\mathbb T_h}
  \left(\frac{H_T}{h_T}\right)^2
  \sum_{i=1}^2 h_{T,i}^2
  |\partial_{\bm r_{T,i}}v|_{H^1(T)}^2, \notag\\
  \mathcal A_{\RT,h}(v,g)^2
  &:=\sum_{T\in\mathbb T_h}
  \left(
    \frac{H_T}{h_T}\sum_{i=1}^2 h_{T,i}
    \|\partial_{\bm r_{T,i}}\nabla v\|_{L^2(T)^2}
    +h_T\|g\|_{L^2(T)}
  \right)^2.
\end{align}
These quantities are used only for rate comparison; unlike $\kappa_{M,h}$, they involve derivatives of the exact solution and are not asserted to be computable from the discrete matrices.

\begin{proposition}[Two-dimensional directional comparison]
\label{prop:anisotropic-directional-comparison}
Let $d=2$, $g_h\in X_h$, and $z:=Jg_h$. Suppose that the two interpolants in \eqref{eq:smooth-anisotropic-comparison} are well defined and that
\begin{align*}
  \mathcal A_{L,h}(z)+\mathcal A_{\RT,h}(z,g_h)<\infty.
\end{align*}
Then, there exists a constant $C_{\rm ani}$, independent of the mesh sizes and aspect ratios, such that
\begin{align}\label{eq:anisotropic-defect-comparison}
  \|\bm{\sigma}_h^M(g_h)-\nabla z_h^c(g_h)\|_{L^2(\Omega)^2}
  \le C_{\rm ani}
  \left(
    \mathcal A_{L,h}(z)^2
    +\mathcal A_{\RT,h}(z,g_h)^2
  \right)^{1/2}.
\end{align}
If these local interpolation hypotheses hold for any $g_h\in X_h$, then
\begin{align}\label{eq:anisotropic-kappa-comparison}
  \kappa_{M,h}
  \le C_{\rm ani}
  \sup_{0\ne g_h\in X_h}
  \frac{
    \left(
      \mathcal A_{L,h}(Jg_h)^2
      +\mathcal A_{\RT,h}(Jg_h,g_h)^2
    \right)^{1/2}}
  {\|g_h\|_{L^2(\Omega)}}.
\end{align}
\end{proposition}

\begin{proof}
The physical-direction Lagrange estimate of
\citet[Theorem~16.7]{Ishizaka2025Interpolation}, with $m=1$, $\ell=2$, and $p=q=2$, uses the directional scales $\alpha_{T,i}=h_{T,i}$ and gives
\begin{align*}
  |z-I_h^cz|_{H^1(\Omega)}
  \le C_L\mathcal A_{L,h}(z).
\end{align*}
Here, the square-sum form follows by applying the local estimate and the Cauchy--Schwarz inequality to the two physical directional contributions. The lowest-order physical-direction RT estimate \citet[Theorem~20.14]{Ishizaka2025Interpolation} (see also \citet[Theorem~2]{Ishizaka2022RT}) uses the same scales $\alpha_{T,i}=h_{T,i}$. In the notation of \citet[Theorem~2]{Ishizaka2022RT}, the corresponding standard-position factor is written using $H_{T_0}$; $H_{T_0}$ and $H_T$ are equivalent up to a factor two by \citet[Lemma~1]{Ishizaka2022RT}, so this difference is absorbed into the mesh-independent constant. Applied to $\nabla z$ and using $\divop(\nabla z)=\varDelta z=-g_h$, the RT estimate gives
\begin{align*}
  \|\nabla z-I_h^{\RT}\nabla z\|_{L^2(\Omega)^2}
  \le C_{\RT}\mathcal A_{\RT,h}(z,g_h).
\end{align*}
Substitution into \eqref{eq:smooth-anisotropic-comparison} proves \eqref{eq:anisotropic-defect-comparison}, after increasing the constant if necessary. Taking the supremum over $0\ne g_h\in X_h$ proves \eqref{eq:anisotropic-kappa-comparison}.
\end{proof}

For comparison with the standard globally regular $O(h)$ theory, assume a two-dimensional semi-regular family satisfying
\begin{align}\label{eq:semi-regular-condition}
  \sup_h\max_{T\in\mathbb T_h}\frac{H_T}{h_T}\le\gamma_0<\infty.
\end{align}
Then $h_{T,i}\le h_T\le h$. For $v\in H^2(T)$ and any unit direction $\bm r$,
\begin{align*}
  |\partial_{\bm r}v|_{H^1(T)}
  &=\|\partial_{\bm r}\nabla v\|_{L^2(T)^2}
  \le \sqrt{2}\,|v|_{H^2(T)},\\
  \|\varDelta v\|_{L^2(T)}
  &\le \sqrt{2}\,|v|_{H^2(T)}.
\end{align*}
Thus, the divergence term $h_T\|\varDelta v\|_{L^2(T)}$ in the RT estimate is absorbed by the same $h|v|_{H^2(T)}$ scale, and the two directional estimates imply
\begin{align}\label{eq:semi-regular-interpolation-comparison}
  |v-I_h^cv|_{H^1(\Omega)}
  +\|\nabla v-I_h^{\RT}\nabla v\|_{L^2(\Omega)^2}
  \le C(\gamma_0)h\,\|v\|_{H^2(\Omega)}
\end{align}
for any $v\in H^2(\Omega)\cap H_0^1(\Omega)$. Therefore, arbitrarily large aspect ratios are allowed in this comparison; what is controlled is the semi-regular parameter $H_T/h_T$. In particular, \eqref{eq:semi-regular-interpolation-comparison} verifies Assumption~\ref{ass:uniform-smooth} for such two-dimensional families.

The same framework gives directional tetrahedral estimates. In three dimensions, \citet[Theorem~16.7]{Ishizaka2025Interpolation}, with $m=1$ and $\ell=2$, requires local $W^{2,p}$ regularity with $p>2$ for the first-order $H^1$ Lagrange estimate; \citet[Theorems~20.14 and~20.15]{Ishizaka2025Interpolation} supply the corresponding RT estimates. These local estimates can be inserted into \eqref{eq:general-anisotropic-comparison} without assuming the global elliptic estimate $\|Jg\|_{H^2(\Omega)}\le C\|g\|_{L^2(\Omega)}$.

For a dimension-independent comparison with the classical globally regular case, we use the following sufficient assumption.

\begin{assumption}[Uniform $H^2$ interpolation comparison]
\label{ass:uniform-smooth}
For a given simplicial mesh family, suppose that there exists a constant $C_{\rm int}$, independent of the mesh and of $h$, such that 
\begin{align}\label{eq:uniform-interpolation-comparison}
  |v-I_h^c v|_{H^1(\Omega)}
  +\|\nabla v-I_h^{\RT}\nabla v\|_{L^2(\Omega)^d}
  \le C_{\rm int}h\,\|v\|_{H^2(\Omega)}
\end{align}
for any $v\in H^2(\Omega)\cap H_0^1(\Omega)$ and any mesh in the family.
\end{assumption}

\begin{corollary}[Recovery of the classical order in the globally regular case]
\label{cor:classical-order}
Assume, only for this comparison, that the Dirichlet Poisson solution operator satisfies
\begin{align}\label{eq:H2-reg-comparison}
  \|Jg\|_{H^2(\Omega)}
  \le C_{\rm reg}\|g\|_{L^2(\Omega)}
  \quad\forall g\in L^2(\Omega),
\end{align}
and that Assumption~\ref{ass:uniform-smooth} holds. We set
\begin{align*}
  C_M:=\sqrt{\pi^{-2}+C_{\rm int}^2C_{\rm reg}^2}.
\end{align*}
Then,
\begin{align}\label{eq:classical-kappa-order}
  \kappa_{M,h}\le C_{\rm int}C_{\rm reg}h,
  \quad
  M_h\le \widehat M_h\le C_Mh,
\end{align}
and consequently
\begin{align}\label{eq:classical-L2-order}
  \|J-P_hJ\|_{\mathcal L(L^2(\Omega),L^2(\Omega))}
  \le C_M^2h^2.
\end{align}
\end{corollary}

\begin{proof}
For $g_h\in X_h$, we set $z:=Jg_h$. From \eqref{eq:smooth-anisotropic-comparison}, \eqref{eq:uniform-interpolation-comparison}, and \eqref{eq:H2-reg-comparison},
\begin{align*}
  \|\bm{\sigma}_h^M(g_h)-\nabla z_h^c(g_h)\|_{L^2(\Omega)^d}
  &\le |z-I_h^c z|_{H^1(\Omega)}
    +\|\nabla z-I_h^{\RT}\nabla z\|_{L^2(\Omega)^d}\\
  &\le C_{\rm int}h\,\|z\|_{H^2(\Omega)}\\
  &\le C_{\rm int}C_{\rm reg}h\,\|g_h\|_{L^2(\Omega)}.
\end{align*}
Taking the supremum over $0\ne g_h\in X_h$ proves the first estimate in \eqref{eq:classical-kappa-order}. The bound $C_{0,h}\le h/\pi$, the inequality $M_h\le\widehat M_h$, and the definition of $\widehat M_h$ give
\begin{align*}
  M_h\le\widehat M_h
  &\le h\sqrt{\pi^{-2}+C_{\rm int}^2C_{\rm reg}^2}
  =C_Mh.
\end{align*}
Finally, Theorem~\ref{thm:L2} yields
\begin{align*}
  \|J-P_hJ\|_{\mathcal L(L^2(\Omega),L^2(\Omega))}
  \le M_h^2
  \le \widehat M_h^2
  \le C_M^2h^2,
\end{align*}
which is \eqref{eq:classical-L2-order}.
\end{proof}

\subsection{Recovery by algebraic corner grading without global \texorpdfstring{$H^2$}{H2}-regularity}
\label{subsec:graded-theory}
On the L-shaped domain, the Dirichlet solution operator does not map $L^2$ into $H^2$, so the preceding globally regular comparison is unavailable. We separate the leading corner singularity and grade the mesh toward the re-entrant corner. The tensor-product family used below has the same local grading law as the usual isotropic construction, but its aspect ratio is unbounded while the semi-regular parameter remains bounded. See \citet{ApelNicaise1998} for more general graded meshes.

Let
\begin{align*}
  \Omega_L:=(-1,1)^2\setminus([0,1]\times[-1,0]),
  \quad
  \alpha:=\frac{\pi}{3\pi/2}=\frac23.
\end{align*}
With polar coordinates centred at the re-entrant corner, we choose $0<r_0<r_1<1$ and a smooth radial cut-off $\eta$ such that $\eta(r)=1$ for $r\le r_0$ and $\eta(r)=0$ for $r\ge r_1$. We then set
\begin{align}\label{eq:corner-singular-function}
  S(r,\theta):=\eta(r)r^\alpha\sin(\alpha\theta).
\end{align}
Because $\sin(\alpha\theta)$ vanishes on the two rays forming the re-entrant corner and $\eta$ is supported away from the remaining boundary, $S\in H_0^1(\Omega_L)$. The corner-singularity decomposition for polygonal Dirichlet problems \citep[Theorem~4.4.3.7 and Section~8.4]{Grisvard1985} gives, for any $g\in L^2(\Omega_L)$,
\begin{align}\label{eq:corner-decomposition}
Jg=z_R+c_gS,
\quad
z_R\in H^2(\Omega_L)\cap H_0^1(\Omega_L).
\end{align}
Furthermore, the continuity of this decomposition yields
\begin{align}
\|z_R\|_{H^2(\Omega_L)}+|c_g|
\le C_{\rm sing}\|g\|_{L^2(\Omega_L)}.
\end{align}

For $q>1$ and $n\in\mathbb N$, let $\mathbb T_n^{(q)}$ be the right-triangle mesh obtained from the logical grid $\xi_j=j/n$, $-n\le j\le n$, through
\begin{align}\label{eq:general-grading-map}
  x_j=\operatorname{sign}(\xi_j)|\xi_j|^q,
  \quad
  y_j=\operatorname{sign}(\xi_j)|\xi_j|^q,
\end{align}
with the lower-right quadrant removed. Each retained rectangle is divided along one of its diagonals into two right triangles. Therefore, $H_T/h_T=2$ for any element, so the family is uniformly semi-regular. For $q>1$, its maximum aspect ratio grows like $n^{q-1}$. We write $h_n:=\max_{T\in\mathbb T_n^{(q)}}h_T$ and use the subscript $n$ for all finite element spaces, operators, and Marini constants associated with this mesh.

\begin{lemma}[Geometry of the graded family]\label{lem:graded-geometry}
Let $\mathcal P_n$ be the finite patch of triangles whose closures contain the re-entrant corner and set $\omega_n:=\bigcup_{T\in\mathcal P_n}T$. Then,
\begin{align}\label{eq:graded-geometry}
  h_n\simeq n^{-1},
  \quad
  \operatorname{diam}(\omega_n)\simeq n^{-q}.
\end{align}
Furthermore, if $T\notin\mathcal P_n$ and $r_T:=\operatorname{dist}(T,0)$, then
\begin{align}\label{eq:graded-local-diameter-lemma}
  r_T\ge c n^{-q},
  \quad
  h_T\le Cn^{-1}r_T^{1-1/q},
\end{align}
where the constants may depend on the fixed grading exponent $q$, but are independent of $n$.
\end{lemma}

\begin{proof}
We fix $q>1$, and set
\begin{align*}
a_j:=\left(\frac{j}{n}\right)^q,
\quad
\Delta_j:=a_{j+1}-a_j,
\quad 0\le j\le n-1.
\end{align*}
By symmetry, the numbers $\Delta_j$ are the one-dimensional mesh widths on both sides of the origin. For $1\le j\le n-1$, the mean-value theorem gives
\begin{align*}
\Delta_j
=\frac{q}{n}\zeta_j^{\,q-1},
\quad
\zeta_j\in
\left(\frac{j}{n},\frac{j+1}{n}\right).
\end{align*}
Because $(j+1)/j\le2$,
\begin{align}\label{eq:graded-width-proof}
\Delta_j
\le C_qn^{-1}a_j^{\,1-1/q}.
\end{align}
For the first interval,
\begin{align*}
\Delta_0=a_1=n^{-q}.
\end{align*}
Let $T\notin\mathcal P_n$, and let
\begin{align*}
R=[x_\ell,x_{\ell+1}]\times[y_m,y_{m+1}]
=:I_x\times I_y
\end{align*}
be the graded rectangle containing $T$. If $R$ does not have the origin as a vertex, then at least one coordinate has absolute value at least $a_1$ throughout $R$. If $R$ has the origin as a vertex and $T\notin\mathcal P_n$, then $T$ is the triangle separated from the origin by the diagonal joining the two adjacent axis vertices. Its distance from the origin is $a_1/\sqrt2$. Thus, in either case,
\begin{align}\label{eq:graded-distance-proof}
r_T\ge\frac{a_1}{\sqrt2}
=\frac{n^{-q}}{\sqrt2}.
\end{align}

We set
\begin{align*}
b_x:=\operatorname{dist}(I_x,0)
=\min_{s\in I_x}|s|,
\quad
b_y:=\operatorname{dist}(I_y,0)
=\min_{s\in I_y}|s|.
\end{align*}
We consider first the side length in the $x$-direction. If $b_x > 0$, then $b_x=a_j$ for some $j\ge1$. Since $b_x\le r_T$, \eqref{eq:graded-width-proof} gives
\begin{align*}
|I_x|
\le C_qn^{-1}b_x^{\,1-1/q}
\le C_qn^{-1}r_T^{\,1-1/q}.
\end{align*}
If $b_x=0$, then $|I_x|=n^{-q}$, and \eqref{eq:graded-distance-proof} gives
\begin{align*}
|I_x|
\le C_qn^{-1}r_T^{\,1-1/q}.
\end{align*}
The same argument applies to $I_y$. Consequently,
\begin{align*}
h_T
\le \left(|I_x|^2+|I_y|^2\right)^{1/2}
\le C_qn^{-1}r_T^{\,1-1/q},
\end{align*}
which proves \eqref{eq:graded-local-diameter-lemma}.

Only the three retained rectangles adjacent to the origin can contain elements of $\mathcal P_n$, and each of these rectangles has side length $n^{-q}$. The patch is therefore contained in $[-n^{-q},n^{-q}]^2$ and contains a triangle of diameter $\sqrt2\,n^{-q}$. Therefore,
\begin{align*}
\operatorname{diam}(\omega_n)\simeq n^{-q}.
\end{align*}

Finally, $\phi'(t)=qt^{q-1}\le q$ on $[0,1]$, so any one-dimensional mesh width is at most $q/n$. Thus, $h_n\le C_qn^{-1}$. Conversely,
\begin{align*}
\frac1n
\le
1-\left(1-\frac1n\right)^q
\le
\frac qn.
\end{align*}
The outermost interval occurs as a side of a retained rectangle and hence gives a triangle whose diameter is at least this interval length. Therefore, $h_n\ge n^{-1}$, and
\begin{align*}
h_n\simeq n^{-1}.
\end{align*}
\end{proof}

We set
\begin{align*}
  \delta_n:=\operatorname{diam}(\omega_n),
\end{align*}
so that Lemma~\ref{lem:graded-geometry} gives
\begin{align*}
  \delta_n\simeq n^{-q},
  \quad
  h_T\lesssim n^{-1}r_T^{1-1/q}
  \quad (T\notin\mathcal P_n),
\end{align*}
where $r_T=\operatorname{dist}(T,0)$. Balancing this local grading law with $|D^2S|\simeq r^{\alpha-2}$ yields the threshold below. The argument uses the grading law and uniform semi-regularity, not the large aspect ratios themselves; shape-regular isotropically graded corner meshes satisfy an analogous local law.

\begin{lemma}[Approximation of the corner singular function]\label{lem:graded-singular-approx}
Let $q>1/\alpha$.  If $I_n^c$ is the nodal $\Pone$ interpolant and $I_n^{\RT}\nabla S$ is the lowest-order RT interpolant defined by its face-flux degrees of freedom, then
\begin{align}\label{eq:graded-singular-approx}
  |S-I_n^cS|_{H^1(\Omega_L)}
  +\|\nabla S-I_n^{\RT}\nabla S\|_{L^2(\Omega_L)^2}
  \le Cn^{-1}.
\end{align}
Here, $C$ may depend on the fixed grading exponent $q$ and the cut-off function $\eta$, but is independent of $n$.
\end{lemma}

\begin{proof}
Recall that
\begin{align*}
  \delta_n =\operatorname{diam}(\omega_n)\simeq n^{-q}
\end{align*}
as in Lemma~\ref{lem:graded-geometry}. We divide the proof into the corner patch $\omega_n$ and the remaining elements.

\medskip
\noindent
\emph{Step 1: regularity of the singular function near the corner.}
Since $\eta(r)=1$ for $r\le r_0$, the singular function has the form
\begin{align*}
  S(r,\theta)=r^\alpha\sin(\alpha\theta),
  \quad \alpha=\frac23,
\end{align*}
and thus
\begin{align*}
  |\nabla S|\le Cr^{\alpha-1},
  \quad
  |D^2S|\le Cr^{\alpha-2}.
\end{align*}
Because $\alpha>0$,
\begin{align*}
  \int_0^{r_0}r^{\alpha-2}r\,dr
  =
  \int_0^{r_0}r^{\alpha-1}\,dr
  <\infty.
\end{align*}
The terms involving derivatives of the cut-off are smooth and supported away from the corner. Since $S$ is smooth in $\Omega_L$ and its classical derivatives up to order two belong to $L^1(\Omega_L)$, they represent its distributional derivatives. Hence,
\begin{align*}
  S\in W^{2,1}(\Omega_L),
  \quad
  \nabla S\in W^{1,1}(\Omega_L)^2.
\end{align*}
Moreover, $S$ extends continuously to $\overline{\Omega_L}$ with $S(0)=0$, so its nodal interpolant is well defined. For every $T\in\mathbb T_n^{(q)}$, the trace theorem for $W^{1,1}(T)$ shows that $\nabla S|_T$ has an $L^1$ trace on each edge. Therefore, the face-flux degrees of freedom defining the lowest-order RT interpolant of $\nabla S$ are also well defined.

The same singular function is not in $H^2(\Omega_L)$. Indeed, in a neighbourhood of the corner where $\eta\equiv1$, we write
\begin{align*}
  S_0(r,\theta)=r^\alpha\sin(\alpha\theta).
\end{align*}
A direct calculation gives
\begin{align*}
  |D^2S_0|_F^2
  =2\alpha^2(1-\alpha)^2r^{2\alpha-4}.
\end{align*}
Therefore, for sufficiently small $\varepsilon>0$,
\begin{align*}
  \int_{\Omega_L\cap B_\varepsilon(0)} |D^2S|_F^2\,dx
  =C_\alpha\int_0^\varepsilon r^{2\alpha-3}\,dr
  =\infty
  \quad (\alpha=2/3).
\end{align*}
Consequently,
\begin{align*}
  S\in W^{2,1}(\Omega_L)\setminus H^2(\Omega_L).
\end{align*}
Thus, the elements touching the re-entrant corner cannot be treated by the usual $H^2$-based interpolation estimate and have to be handled separately.

\medskip
\noindent
\emph{Step 2: approximation on the corner patch.}
For all sufficiently large $n$, the cut-off satisfies $\eta\equiv1$ on $\omega_n$, as observed in Step~1. Thus, $S$ agrees there with the homogeneous function $S_0$ defined above.
Set $\varepsilon_n:=n^{-q}$. Every $T\in\mathcal P_n$ has the origin as a vertex and lies in one of the three retained grid squares adjacent to the origin. Hence,
\begin{align*}
  T=\varepsilon_n\widehat T
\end{align*}
for some $\widehat T$ in the fixed finite family $\widehat{\mathcal P}$ of possible corner triangles on the corresponding unit grid. In particular, $\varepsilon_n\simeq\delta_n$. For $\bm x=\varepsilon_n\widehat{\bm x}\in T$, homogeneity gives
\begin{align*}
  S(\bm x)
  =
  S_0(\varepsilon_n\widehat{\bm x})
  =
  \varepsilon_n^\alpha S_0(\widehat{\bm x}).
\end{align*}
The affine covariance of the nodal interpolant therefore gives
\begin{align*}
  (I_T^cS)(\varepsilon_n\widehat{\bm x})
  =
  \varepsilon_n^\alpha
  (I_{\widehat T}^cS_0)(\widehat{\bm x}),
\end{align*}
and the change of variables $\bm x=\varepsilon_n\widehat{\bm x}$ yields
\begin{align*}
  |S-I_T^cS|_{H^1(T)}
  =
  \varepsilon_n^\alpha
  |S_0-I_{\widehat T}^cS_0|_{H^1(\widehat T)}
  \le C\varepsilon_n^\alpha.
\end{align*}
Likewise,
\begin{align*}
  \nabla S(\varepsilon_n\widehat{\bm x})
  =
  \varepsilon_n^{\alpha-1}
  \nabla S_0(\widehat{\bm x}),
\end{align*}
and the contravariant Piola covariance of the lowest-order RT interpolation gives
\begin{align*}
  (I_T^{\RT}\nabla S)(\varepsilon_n\widehat{\bm x})
  =
  \varepsilon_n^{\alpha-1}
  (I_{\widehat T}^{\RT}\nabla S_0)(\widehat{\bm x}).
\end{align*}
Consequently,
\begin{align*}
  \|\nabla S-I_T^{\RT}\nabla S\|_{L^2(T)^2}
  &=
  \varepsilon_n^\alpha
  \|\nabla S_0-I_{\widehat T}^{\RT}\nabla S_0\|_
       {L^2(\widehat T)^2}
  \le C\varepsilon_n^\alpha.
\end{align*}
Here, $S_0\in H^1(\widehat T)\cap \mathcal{C}^0(\overline{\widehat T})$ and $\nabla S_0\in W^{1,1}(\widehat T)^2\cap L^2(\widehat T)^2$, so the reference-element errors are finite. They are bounded uniformly because $\widehat{\mathcal P}$ is finite. Since the number of corner elements is also bounded independently of $n$ and $\varepsilon_n\simeq\delta_n$, we obtain
\begin{align}\label{eq:corner-patch-bound}
  |S-I_n^cS|_{H^1(\omega_n)}
  +\|\nabla S-I_n^{\RT}\nabla S\|_{L^2(\omega_n)^2}
  \le C\delta_n^\alpha.
\end{align}
Using $\delta_n\simeq n^{-q}$ and $q\alpha>1$,
\begin{align*}
  \delta_n^\alpha
  \le Cn^{-q\alpha}
  \le Cn^{-1}.
\end{align*}
The finitely many smaller values of $n$ are covered by increasing the constant $C$.

\medskip
\noindent
\emph{Step 3: approximation away from the corner.}
Let $T\notin\mathcal P_n$ and recall that
\begin{align*}
  r_T =\operatorname{dist}(T,0).
\end{align*}
From Lemma~\ref{lem:graded-geometry},
\begin{align*}
  r_T\ge cn^{-q},
  \quad
  h_T\le Cn^{-1}r_T^{1-1/q}.
\end{align*}
Consequently,
\begin{align*}
  \frac{h_T}{r_T}
  \le Cn^{-1}r_T^{-1/q}
  \le C,
\end{align*}
and hence
\begin{align*}
  h_T\le Cr_T.
\end{align*}
More precisely, for any $\bm x\in T$,
\begin{align*}
  r_T
  \le |\bm x|
  \le r_T+h_T
  \le Cr_T.
\end{align*}
Thus, the radial variable $r=|\bm x|$ is comparable with $r_T$ on each non-corner element. Furthermore, $S|_T\in H^2(T)$ and $\nabla S|_T\in H^1(T)^2$ because $r_T>0$. The elements are right triangles and satisfy $H_T/h_T=2$, while $h_{T,i}\le h_T$. The directional Lagrange estimate \citep[Theorem~16.7]{Ishizaka2025Interpolation} (see also \citealp{IshizakaKobayashiTsuchiya2023Interpolation}), specialised to the $L^2$ case with $m=1$ and $\ell=2$, therefore gives
\begin{align*}
  |S-I_n^cS|_{H^1(T)}
  &\le
  C\frac{H_T}{h_T}
  \sum_{i=1}^2
  h_{T,i}|\partial_{\bm r_{T,i}}S|_{H^1(T)}
  \le
  Ch_T\|D^2S\|_{L^2(T)^{2\times2}}.
\end{align*}
Similarly, the lowest-order RT estimate \citep[Theorem~20.14]{Ishizaka2025Interpolation} (see also \citealp{Ishizaka2022RT}), with $k=\ell=0$, gives
\begin{align*}
  \|\nabla S-I_n^{\RT}\nabla S\|_{L^2(T)^2}
  &\le
  C\frac{H_T}{h_T}
  \sum_{i=1}^2
  h_{T,i}
  \|\partial_{\bm r_{T,i}}\nabla S\|_{L^2(T)^2}
  +Ch_T\|\varDelta S\|_{L^2(T)}
  \\
  &\le
  Ch_T\|D^2S\|_{L^2(T)^{2\times2}},
\end{align*}
where $\|\varDelta S\|_{L^2(T)}\le \sqrt{2}\,\|D^2S\|_{L^2(T)^{2\times2}}$ was used in the last step. Consequently,
\begin{align*}
  |S-I_n^cS|_{H^1(T)}^2
  +\|\nabla S-I_n^{\RT}\nabla S\|_{L^2(T)^2}^2
  \le
  Ch_T^2\|D^2S\|_{L^2(T)^{2\times2}}^2.
\end{align*}
Near the corner,
\begin{align*}
  |D^2S|\le Cr^{\alpha-2},
\end{align*}
whereas in the fixed cut-off region away from the corner the second derivatives of $S$ are uniformly bounded. Moreover, $r_T\ge c\delta_n$ for every non-corner element, and $S$ vanishes for $r\ge r_1<1$. Consequently, using the local diameter estimate and polar coordinates,
\begin{align*}
 \sum_{T\notin\mathcal P_n}
 h_T^2\|D^2S\|_{L^2(T)^{2\times2}}^2
 &\le
 Cn^{-2}
 \left(
 1+
 \int_{c\delta_n}^{1}
 r^{2(1-1/q)}
 r^{2(\alpha-2)}
 r\,dr
 \right)
 \\
 & =
 Cn^{-2}
 \left(
 1+
 \int_{c\delta_n}^{1}
 r^{2\alpha-1-2/q}\,dr
 \right).
\end{align*}
The last integral is bounded uniformly in $n$ precisely when
\begin{align*}
  2\alpha-1-\frac2q>-1,
\end{align*}
or equivalently,
\begin{align*}
  \alpha>\frac1q.
\end{align*}
This is exactly the assumption $q>1/\alpha$.  Therefore,
\begin{align*}
  A_n^2+B_n^2
  \le Cn^{-2},
\end{align*}
where
\begin{align*}
  A_n
  &:=
  |S-I_n^cS|_{H^1(\Omega_L\setminus\omega_n)}, \quad
  B_n
  :=
  \|\nabla S-I_n^{\RT}\nabla S\|_
       {L^2(\Omega_L\setminus\omega_n)^2}.
\end{align*}
It follows that
\begin{align*}
  A_n+B_n
  \le
  \sqrt{2}\,(A_n^2+B_n^2)^{1/2}
  \le Cn^{-1}.
\end{align*}
Combining this estimate with \eqref{eq:corner-patch-bound} proves \eqref{eq:graded-singular-approx}.
\end{proof}

\begin{theorem}[Recovery of the computable order on the L-shaped domain]\label{thm:graded-Lshape}
Let $\Omega=\Omega_L$, $\alpha=2/3$, and let $\{\mathbb T_n^{(q)}\}$ be the graded family above with
\begin{align*}
  q>\frac1\alpha=\frac32.
\end{align*}
Then, there exists $C>0$, independent of $n$, such that
\begin{align}\label{eq:graded-kappa-order}
  \kappa_{M,n}\le Ch_n,
  \quad
  \widehat M_n\le Ch_n,
\end{align}
and both discrete solution operators satisfy
\begin{align}\label{eq:graded-L2-both}
  \|J-P_nJ\|_{\cL(L^2(\Omega_L),L^2(\Omega_L))}
  +\|J-J_n^{\CR}\|_{\cL(L^2(\Omega_L),L^2(\Omega_L))}
  \le Ch_n^2.
\end{align}
In particular, the quadratic grading $q=2$ used in Section~\ref{subsec:lshape-graded} is admissible.
\end{theorem}

\begin{proof}
Let $g_n\in X_n$ and set $z:=Jg_n$. We use the corner decomposition to construct an admissible conforming--RT pair for Corollary~\ref{cor:joint-hypercircle}.

\medskip
\noindent
\emph{Step 1: approximation of the regular and singular parts.}
By the corner decomposition \eqref{eq:corner-decomposition},
\begin{align*}
  z=z_R+c_{g_n}S,
  \quad
  \|z_R\|_{H^2(\Omega_L)}+|c_{g_n}|
  \le C\|g_n\|_{L^2(\Omega_L)}.
\end{align*}
Because $z_R\in H^2(\Omega_L)$, the semi-regular interpolation estimates on the present right-triangle family give
\begin{align}\label{eq:graded-regular-approx}
  |z_R-I_n^cz_R|_{H^1(\Omega_L)}
  +\|\nabla z_R-I_n^{\RT}\nabla z_R\|_{L^2(\Omega_L)^2}
  \le Ch_n\|z_R\|_{H^2(\Omega_L)}.
\end{align}
For the singular part, Lemma~\ref{lem:graded-singular-approx} gives
\begin{align*}
  |S-I_n^cS|_{H^1(\Omega_L)}
  +\|\nabla S-I_n^{\RT}\nabla S\|_{L^2(\Omega_L)^2}
  \le Cn^{-1}.
\end{align*}
Since $h_n\simeq n^{-1}$ by Lemma~\ref{lem:graded-geometry}, this is bounded by $Ch_n$.  By the linearity of both interpolation operators and the stability of the corner decomposition, we therefore obtain
\begin{align}\label{eq:graded-full-approx}
  |z-I_n^cz|_{H^1(\Omega_L)}
  +\|\nabla z-I_n^{\RT}\nabla z\|_{L^2(\Omega_L)^2}
  \le Ch_n\|g_n\|_{L^2(\Omega_L)}.
\end{align}
The conforming interpolant $I_n^cz$ is well defined. Indeed, $z_R\in H^2(\Omega_L)\hookrightarrow \mathcal{C}^0(\overline{\Omega_L})$ in two dimensions and $S$ is continuous. Furthermore, $z\in H_0^1(\Omega_L)\cap \mathcal{C}^0(\overline{\Omega_L})$, so its boundary nodal values vanish. Therefore,
\begin{align}\label{eq:graded-conforming-admissible}
  I_n^cz\in V_n^c.
\end{align}

\medskip
\noindent
\emph{Step 2: the RT interpolant is an equilibrated flux.}
The proof of Lemma~\ref{lem:graded-singular-approx} shows that
\begin{align*}
  S\in W^{2,1}(\Omega_L),
  \quad
  \nabla S\in W^{1,1}(\Omega_L)^2.
\end{align*}
Because $z_R\in H^2(\Omega_L)$ and $\Omega_L$ is bounded,
\begin{align*}
  \nabla z_R\in W^{1,1}(\Omega_L)^2.
\end{align*}
Consequently,
\begin{align}\label{eq:graded-global-W11}
  \nabla z\in W^{1,1}(\Omega_L)^2.
\end{align}
In particular, the face traces of $\nabla z$ are well defined and agree from the two sides of every interior face. Thus, the canonical lowest-order RT face-flux degrees of freedom are globally compatible. Let $T\in\mathbb T_n^{(q)}$. Since $I_n^{\RT}\nabla z|_T\in\RT^0(T)$, its divergence is constant on $T$. Using Green's formula and preservation of the RT face fluxes gives
\begin{align*}
  |T|\,\divop(I_n^{\RT}\nabla z)|_T
  &=
  \int_T\divop(I_n^{\RT}\nabla z)\,dx
  =
  \sum_{F\subset\partial T}
  \int_F I_n^{\RT}\nabla z\cdot\bm n_T\,ds
\\
  &=
  \sum_{F\subset\partial T}
  \int_F \nabla z\cdot\bm n_T\,ds
 =
  \int_T\varDelta z\,dx.
\end{align*}
Because $z=Jg_n$ satisfies $-\varDelta z=g_n$ and $g_n$ is constant on each element,
\begin{align*}
  \divop(I_n^{\RT}\nabla z)|_T=-g_n|_T.
\end{align*}
Furthermore, \eqref{eq:graded-global-W11} implies that the face fluxes used on the two elements adjacent to an interior face coincide. Therefore, the elementwise RT interpolants assemble into a globally normal-continuous field, and hence
\begin{align}\label{eq:graded-equilibrated-interpolant}
  I_n^{\RT}\nabla z\in\mathcal Q_n(g_n).
\end{align}

\medskip
\noindent
\emph{Step 3: hypercircle minimisation and the computable constant.}
By \eqref{eq:graded-conforming-admissible} and \eqref{eq:graded-equilibrated-interpolant}, the pair
\begin{align*}
  (I_n^cz,I_n^{\RT}\nabla z)
  \in V_n^c\times\mathcal Q_n(g_n)
\end{align*}
is admissible in Corollary~\ref{cor:joint-hypercircle}. Therefore,
\begin{align*}
  \|\bm\sigma_n^M(g_n)-\nabla z_n^c(g_n)\|_{L^2(\Omega_L)^2}
  &\le
  \|I_n^{\RT}\nabla z-\nabla I_n^cz\|_{L^2(\Omega_L)^2}
\\
  &\le
  \|\nabla z-I_n^{\RT}\nabla z\|_{L^2(\Omega_L)^2}
  +|z-I_n^cz|_{H^1(\Omega_L)}
\\
  &\le
  Ch_n\|g_n\|_{L^2(\Omega_L)},
\end{align*}
where the last inequality follows from \eqref{eq:graded-full-approx}. Taking the supremum over $0\ne g_n\in X_n$ yields
\begin{align*}
  \kappa_{M,n}\le Ch_n.
\end{align*}
By the definition of the diameter-explicit constant,
\begin{align*}
  \widehat M_n^2
  =
  \frac{h_n^2}{\pi^2}+\kappa_{M,n}^2
  \le Ch_n^2,
\end{align*}
and hence
\begin{align*}
  \widehat M_n\le Ch_n.
\end{align*}

Finally, the $L^2$ estimates of the conforming and the CR problems proved above give,
for every $f\in L^2(\Omega_L)$,
\begin{align*}
  \|(J-P_nJ)f\|_{L^2(\Omega_L)}
  &\le
  \widehat M_n^2\|f\|_{L^2(\Omega_L)},\\
  \|(J-J_n^{\CR})f\|_{L^2(\Omega_L)}
  &\le
  \widehat M_n^2\|f\|_{L^2(\Omega_L)}.
\end{align*}
Dividing by $\|f\|_{L^2(\Omega_L)}$ and taking the supremum over $0\ne f\in L^2(\Omega_L)$ gives
\begin{align*}
  \|J-P_nJ\|_{\cL(L^2(\Omega_L),L^2(\Omega_L))}
  +
  \|J-J_n^{\CR}\|_{\cL(L^2(\Omega_L),L^2(\Omega_L))}
  \le Ch_n^2.
\end{align*}
This proves \eqref{eq:graded-kappa-order} and
\eqref{eq:graded-L2-both}.
\end{proof}

\begin{remark}[Critical grading and scope]
The condition $q>1/\alpha$ is exactly the threshold for uniform boundedness of the radial integral in the estimate above. At the critical value $q=1/\alpha$, the radial integral grows like $\log n$. Because $h_n\simeq n^{-1}$, the same argument gives
\begin{align*}
  &\kappa_{M,n}
  +\widehat M_n
  =
  O\!\left(h_n\sqrt{|\log h_n|}\right),\\
  &\|J-P_nJ\|_{\cL(L^2,L^2)}
  +\|J-J_n^{\CR}\|_{\cL(L^2,L^2)}
  =
  O\!\left(h_n^2|\log h_n|\right).
\end{align*}
The theorem is a mesh-family result for this corner and this grading mechanism. It does not assert a general aspect-ratio-independent rate theorem for arbitrary non-convex domains or graded meshes.
\end{remark}

\section{Numerical experiments}\label{sec:numerics}
We consider uniform and stretched square meshes, quasi-uniform and graded meshes on the L-shaped domain, and two tetrahedral mesh families. The computations use double precision, sparse scalar solves, and a largest-eigenvalue iteration. The majorants in the tables are ordinary floating-point evaluations, not verified enclosures; see Remark~\ref{rem:verified-matrix}. Rates are computed from the actual mesh diameters,
\begin{align*}
  \operatorname{rate}(Q_j)
  :=\frac{\log(Q_{j-1}/Q_j)}{\log(h_{j-1}/h_j)},
\end{align*}
not with the refinement parameter $n$ itself.

Figure~\ref{fig:mesh-families-2d} shows representative meshes from the two-dimensional families.  Boundary-corner diagonals are reversed when necessary so that every triangle has an interior vertex. The vertex set itself is unchanged.
\begin{figure}[htbp]
\centering
\includegraphics[width=\textwidth]{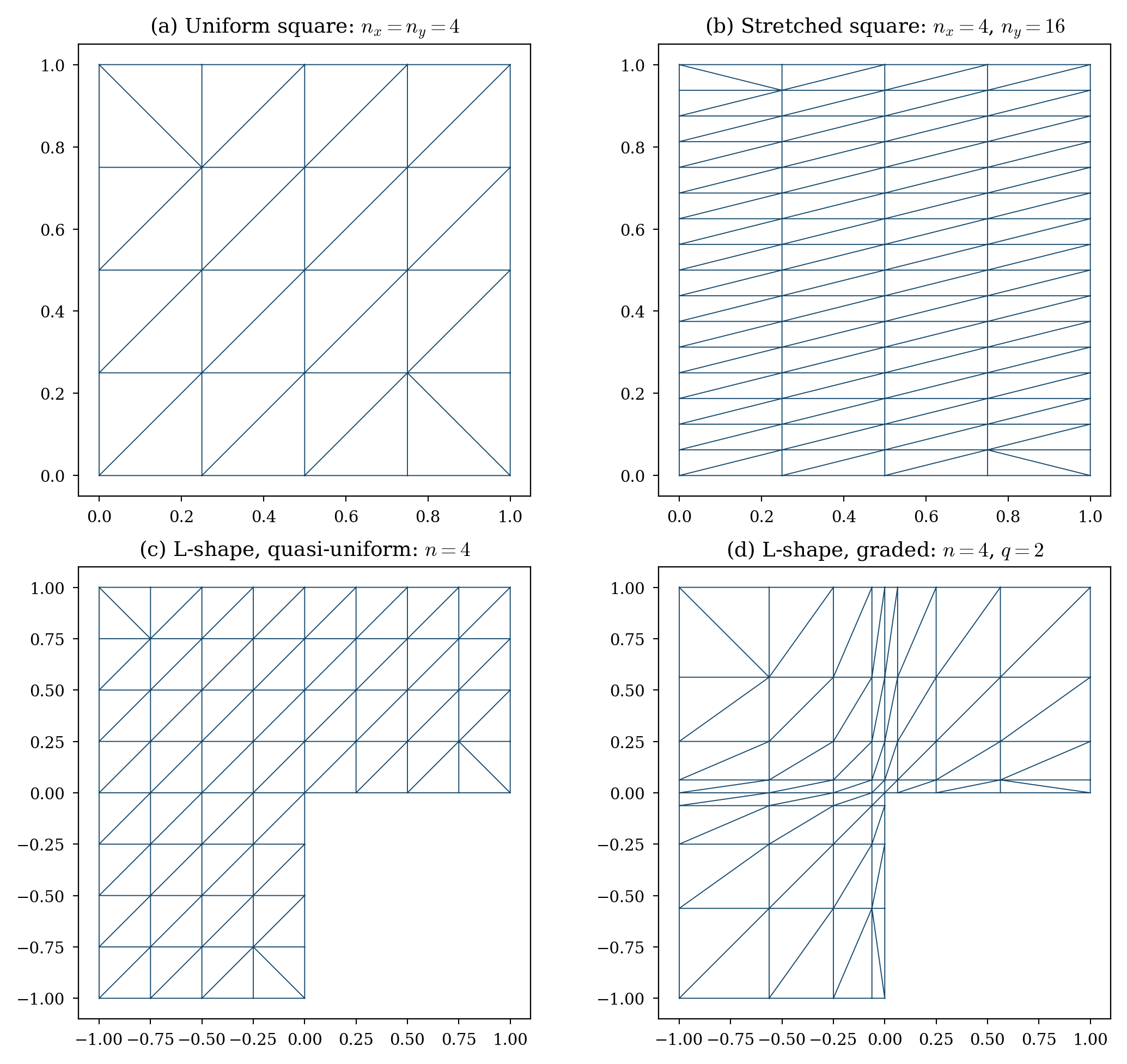}
\caption{Representative two-dimensional meshes: (a) a uniform triangular mesh of the unit square; (b) a vertically stretched mesh with $n_y=n_x^2$; (c) a quasi-uniform mesh of the L-shaped domain; and (d) an algebraically graded L-shaped mesh with $q=2$. Coarser levels from the same mesh families are shown for clarity. Boundary corner diagonals are reversed to avoid triangles whose three vertices lie on the boundary.}
\label{fig:mesh-families-2d}
\end{figure}

For the error test, we take $\Omega=(0,1)^2$ and
\begin{align}\label{eq:numerical-exact-solution}
  u(x,y)=\sin(\pi x)\sin(\pi y),
  \quad
  f(x,y)=2\pi^2\sin(\pi x)\sin(\pi y).
\end{align}
Both the conforming $\Pone$ solution and the CR solution use the unprojected right-hand side $f$. We use the fully explicit diameter-based quantities from Corollaries~\ref{cor:diameter-bound} and \ref{cor:CR-diameter}, namely
\begin{align}\label{eq:numerical-majorant}
  \widehat M_h
  =\sqrt{(h/\pi)^2+\kappa_{M,h}^2},
  \quad
  \widehat\eta^c_{L^2,h}
  :=\widehat M_h\widehat{\mathfrak E}_h(f),
  \quad
  \widehat\eta^{\CR}_{L^2,h}
  :=\widehat M_h^2\|f\|_{L^2(\Omega)}.
\end{align}
The corresponding effectivity indices are
\begin{align}\label{eq:numerical-effectivity}
  \operatorname{eff}_h^c
  &:=
  \frac{\widehat\eta^c_{L^2,h}}
       {\|u-u_h\|_{L^2(\Omega)}},
  \quad
  \operatorname{eff}_h^{\CR}
  :=
  \frac{\widehat\eta^{\CR}_{L^2,h}}
       {\|u-u_h^{\CR}\|_{L^2(\Omega)}}.
\end{align}
For all meshes in the tables, the first branch of \eqref{eq:CR-explicit}, $\widehat M_h^2 \|f \|_{L^2(\Omega)}$, gives the smaller value. All two-dimensional element means, Galerkin right-hand sides, and $L^2$ errors are computed using a tensor-product Gauss--Duffy quadrature rule with twelve Gauss points in each transformed coordinate. Repeating the smooth-square computations with twenty points in each direction leaves the displayed values unchanged. For the L-shaped tests, the singular vertex in the Duffy transformation is chosen as the local mesh vertex nearest the re-entrant corner.

\subsection{Smooth square: uniform and strongly stretched meshes}
For the uniform family the unit square is divided into an $n\times n$ Cartesian grid and triangulated as in Figure~\ref{fig:mesh-families-2d}. Table~\ref{tab:uniform-constants} shows $\gamma_{\rm geo,h}\simeq h$. The CR--conforming gap is the larger contribution, and the observed rate of $\kappa_{M,h}$ tends toward one.

\begin{table}[htbp]
\centering
\caption{Uniform meshes: decomposition of the computable constant.}
\label{tab:uniform-constants}
\small
\begin{tabular}{rrrrrrrr}
\toprule
$n$ & $N_T$ & $h$ & $\gamma_{\rm gap,h}$ &
$\gamma_{\rm geo,h}$ & $\kappa_{M,h}$ & $\widehat M_h$ &
$\kappa_{M,h}/h$\\
\midrule
16  & 512   & 0.088388 & 0.028041 & 0.010417 & 0.029913 & 0.041066 & 0.338432\\
32  & 2048  & 0.044194 & 0.014382 & 0.005208 & 0.015296 & 0.020781 & 0.346116\\
64  & 8192  & 0.022097 & 0.007279 & 0.002604 & 0.007731 & 0.010452 & 0.349873\\
128 & 32768 & 0.011049 & 0.003661 & 0.001302 & 0.003886 & 0.005241 & 0.351731\\
\bottomrule
\end{tabular}
\end{table}

The $L^2$ rates in Table~\ref{tab:uniform-error} approach two for both methods. The conforming majorant has the smaller effectivity index. For CR we use the operator bound $\widehat M_h^2\|f\|_{L^2(\Omega)}$.

\begin{table}[htbp]
\centering
\caption{Uniform meshes: conforming and CR $L^2$-errors and floating-point evaluations of the majorant formulas.}
\label{tab:uniform-error}
\small
\begin{tabular}{crrrrrr}
\toprule
method & $n$ & $L^2$-error & error rate & majorant & majorant rate & effectivity\\
\midrule
$\Pone$ & 16  & 5.374693e-3 & --    & 1.284358e-2 & --    & 2.390\\
$\Pone$ & 32  & 1.350385e-3 & 1.993 & 3.230069e-3 & 1.991 & 2.392\\
$\Pone$ & 64  & 3.379914e-4 & 1.998 & 8.092913e-4 & 1.997 & 2.394\\
$\Pone$ & 128 & 8.452208e-5 & 2.000 & 2.025051e-4 & 1.999 & 2.396\\
\midrule
$\CR$ & 16  & 1.941345e-3 & --    & 1.664398e-2 & --    & 8.573\\
$\CR$ & 32  & 4.861153e-4 & 1.998 & 4.262391e-3 & 1.965 & 8.768\\
$\CR$ & 64  & 1.215742e-4 & 1.999 & 1.078199e-3 & 1.983 & 8.869\\
$\CR$ & 128 & 3.039633e-5 & 2.000 & 2.711202e-4 & 1.992 & 8.920\\
\bottomrule
\end{tabular}
\end{table}

\paragraph{Strongly stretched meshes.}
To increase the aspect ratio on the same domain, we use $n$ subdivisions in $x$ and $n^2$ in $y$, with the same corner-cell modification. Thus,
\begin{align*}
  h_x=n^{-1},\quad h_y=n^{-2},\quad h_x/h_y=n,
\end{align*}
while the maximum diameter still tends to zero. Over the levels in Table~\ref{tab:stretched-constants}, $\kappa_{M,h}/h$ remains bounded. Again the geometric term is proportional to $h$, while the CR--conforming gap is larger.

\begin{table}[htbp]
\centering
\caption{Strongly stretched meshes: spectral decomposition.  Here
$r=h_x/h_y=n$.}
\label{tab:stretched-constants}
\small
\begin{tabular}{rrrrrrrrr}
\toprule
$n$ & $r$ & $N_T$ & $h$ & $\gamma_{\rm gap,h}$ &
$\gamma_{\rm geo,h}$ & $\kappa_{M,h}$ & $\widehat M_h$ &
$\kappa_{M,h}/h$\\
\midrule
2  & 2  & 16   & 0.559017 & 0.131023 & 0.065881 & 0.146654 & 0.230587 & 0.262342\\
4  & 4  & 128  & 0.257694 & 0.077965 & 0.030370 & 0.083671 & 0.117172 & 0.324693\\
8  & 8  & 1024 & 0.125973 & 0.040960 & 0.014846 & 0.043567 & 0.059211 & 0.345848\\
16 & 16 & 8192 & 0.062622 & 0.020744 & 0.007380 & 0.022018 & 0.029700 & 0.351595\\
\bottomrule
\end{tabular}
\end{table}

The $L^2$ rates in Table~\ref{tab:stretched-error} are close to two. These right-triangle meshes remain uniformly semi-regular even though their aspect ratios increase, so they fall under the smooth comparison in Subsection~\ref{subsec:directional-anisotropic-comparison}. This test is not meant as a rate statement for arbitrary anisotropic families.

\begin{table}[htbp]
\centering
\caption{Strongly stretched meshes: conforming and CR $L^2$-errors and
floating-point evaluations of the majorant formulas.}
\label{tab:stretched-error}
\small
\begin{tabular}{crrrrrr}
\toprule
method & $n$ & $L^2$-error & error rate & majorant & majorant rate & effectivity\\
\midrule
$\Pone$ & 2  & 1.486889e-1 & --    & 4.733376e-1 & --    & 3.183\\
$\Pone$ & 4  & 4.191689e-2 & 1.635 & 1.129733e-1 & 1.850 & 2.695\\
$\Pone$ & 8  & 1.018888e-2 & 1.976 & 2.753151e-2 & 1.973 & 2.702\\
$\Pone$ & 16 & 2.522934e-3 & 1.997 & 6.717866e-3 & 2.018 & 2.663\\
\midrule
$\CR$ & 2  & 6.152776e-2 & --    & 5.247682e-1 & --    & 8.529\\
$\CR$ & 4  & 1.265384e-2 & 2.042 & 1.355024e-1 & 1.748 & 10.708\\
$\CR$ & 8  & 2.859204e-3 & 2.078 & 3.460284e-2 & 1.907 & 12.102\\
$\CR$ & 16 & 6.946319e-4 & 2.024 & 8.706017e-3 & 1.974 & 12.533\\
\bottomrule
\end{tabular}
\end{table}

\subsection{L-shaped domain: quasi-uniform and graded meshes}
\label{subsec:lshape}\label{subsec:lshape-graded}
We consider the standard L-shaped polygon
\begin{align*}
  \Omega=(-1,1)^2\setminus\bigl([0,1]\times[-1,0]\bigr),
\end{align*}
whose re-entrant angle at the origin is $3\pi/2$.  With polar coordinates $(r,\theta)$, $0<\theta<3\pi/2$, we set
\begin{align*}
  \alpha:=\frac{2}{3},\quad
  s(r,\theta):=r^\alpha\sin(\alpha\theta),\quad
  \chi(x,y):=(1-x^2)(1-y^2),
\end{align*}
and choose the manufactured solution
\begin{align}\label{eq:lshape-solution}
  u(x,y)=\chi(x,y)s(r,\theta),
  \quad
  f=-\varDelta u
     =-(\varDelta\chi)s-2\nabla\chi\cdot\nabla s.
\end{align}
Because $s$ is harmonic and $\nabla\chi=O(r)$ at the origin, $f=O(r^\alpha)\in L^2(\Omega)$.  In contrast, $\chi=1+O(r^2)$, $\nabla\chi=O(r)$, and $D^2\chi=O(1)$ imply
\begin{align*}
  D^2u=D^2s+O(r^\alpha).
\end{align*}
Furthermore, with $|\cdot|_{\rm F}$ denoting the Frobenius norm,
\begin{align*}
  |D^2s|_{\rm F}^2
  =2\alpha^2(1-\alpha)^2r^{2\alpha-4}.
\end{align*}
Consequently, $|D^2u|_{\rm F}\ge c r^{\alpha-2}$ for all sufficiently small $r>0$, and
\begin{align*}
  \int_{\Omega\cap B_\varepsilon(0)}|D^2u|_{\rm F}^2\,dx
  \ge c\int_0^\varepsilon r^{2\alpha-4}r\,dr
  =c\int_0^\varepsilon r^{-5/3}\,dr=\infty.
\end{align*}
Therefore,
\begin{align*}
  u\in H_0^1(\Omega),\quad -\varDelta u=f\in L^2(\Omega),\quad
  u\notin H^2(\Omega),
\end{align*}
which is the graph-space setting of Theorem~\ref{thm:L2}. A tensor Gauss--Duffy rule of order twelve is used; increasing it to twenty does not change the observed convergence behaviour.

\paragraph{Quasi-uniform meshes.}
The quasi-uniform family is obtained from the grid points $x_j=y_j=j/n$, $-n\le j\le n$, by deleting the lower-right quadrant and triangulating the retained squares as above. In Table~\ref{tab:lshape-quasi-constants}, the observed rate of $\kappa_{M,h}$ moves toward $\alpha=2/3$. Because the geometric term is still first order, the slower decay is due to the CR--conforming gap.

\begin{table}[htbp]
\centering
\caption{L-shaped domain, quasi-uniform meshes: spectral decomposition.}
\label{tab:lshape-quasi-constants}
\small
\begin{tabular}{rrrrrrrr}
\toprule
$n$ & $N_T$ & $h$ & $\gamma_{\rm gap,h}$ &
$\gamma_{\rm geo,h}$ & $\kappa_{M,h}$ & observed rate & $\widehat M_h$\\
\midrule
16  & 1536  & 0.088388 & 0.052799 & 0.010417 & 0.053817 & --    & 0.060728\\
32  & 6144  & 0.044194 & 0.032834 & 0.005208 & 0.033244 & 0.695 & 0.036098\\
64  & 24576 & 0.022097 & 0.020534 & 0.002604 & 0.020699 & 0.684 & 0.021861\\
128 & 98304 & 0.011049 & 0.012885 & 0.001302 & 0.012951 & 0.676 & 0.013420\\
\bottomrule
\end{tabular}
\end{table}

The rates in Table~\ref{tab:lshape-quasi-error} are consistent with the corner-limited value $2\alpha=4/3$. No global $H^2$-regularity constant is used.

\begin{table}[htbp]
\centering
\caption{L-shaped domain, quasi-uniform meshes: conforming and CR $L^2$-errors
and floating-point evaluations of the majorant formulas.}
\label{tab:lshape-quasi-error}
\small
\begin{tabular}{crrrrrr}
\toprule
method & $n$ & $L^2$-error & error rate & majorant & majorant rate & effectivity\\
\midrule
$\Pone$ & 16  & 3.989456e-3 & --    & 1.515088e-2 & --    & 3.798\\
$\Pone$ & 32  & 1.392131e-3 & 1.519 & 5.529142e-3 & 1.454 & 3.972\\
$\Pone$ & 64  & 5.116858e-4 & 1.444 & 2.079540e-3 & 1.411 & 4.064\\
$\Pone$ & 128 & 1.944282e-4 & 1.396 & 7.979249e-4 & 1.382 & 4.104\\
\midrule
$\CR$ & 16  & 3.542365e-3 & --    & 1.691940e-2 & --    & 4.776\\
$\CR$ & 32  & 1.297884e-3 & 1.449 & 5.978367e-3 & 1.501 & 4.606\\
$\CR$ & 64  & 4.924694e-4 & 1.398 & 2.192570e-3 & 1.447 & 4.452\\
$\CR$ & 128 & 1.906033e-4 & 1.369 & 8.262541e-4 & 1.408 & 4.335\\
\bottomrule
\end{tabular}
\end{table}

\paragraph{Corner-graded meshes.}
To test recovery by grading, let
$\xi_j=j/n$, $-n\le j\le n$, and map the logical coordinates by
\begin{align}\label{eq:grading-map}
  x_j=\operatorname{sign}(\xi_j)|\xi_j|^2,
  \quad
  y_j=\operatorname{sign}(\xi_j)|\xi_j|^2.
\end{align}
After deleting the lower-right quadrant, the retained rectangles are triangulated as above.  Triangles near the coordinate axes become strongly anisotropic, measured by
\begin{align*}
  \mathrm{AR}_h:=\max_{T\in\mathbb T_h}\frac{h_T}{a_{T,\min}},
\end{align*}
where $a_{T,\min}$ is the minimum altitude of $T$.

Here, $q=2>1/\alpha=3/2$. Theorem~\ref{thm:graded-Lshape} gives $\kappa_{M,h}=O(h)$. The observed rates in Table~\ref{tab:lshape-graded-constants} are close to one, while $\mathrm{AR}_h$ increases from about $31$ to $255$.

\begin{table}[htbp]
\centering
\caption{L-shaped domain, graded anisotropic meshes: spectral decomposition.}
\label{tab:lshape-graded-constants}
\small
\begin{tabular}{rrrrrrrrr}
\toprule
$n$ & $N_T$ & $h$ & $\mathrm{AR}_h$ & $\gamma_{\rm gap,h}$ &
$\gamma_{\rm geo,h}$ & $\kappa_{M,h}$ & observed rate & $\widehat M_h$\\
\midrule
16  & 1536  & 0.171252 & 31.03  & 0.044779 & 0.020182 & 0.048370 & --    & 0.072877\\
32  & 6144  & 0.087007 & 63.02  & 0.024087 & 0.010254 & 0.025866 & 0.924 & 0.037895\\
64  & 24576 & 0.043849 & 127.01 & 0.012784 & 0.005168 & 0.013677 & 0.930 & 0.019542\\
128 & 98304 & 0.022011 & 255.00 & 0.006666 & 0.002594 & 0.007113 & 0.949 & 0.009984\\
\bottomrule
\end{tabular}
\end{table}

The $L^2$ rates in Table~\ref{tab:lshape-graded-error} are close to two for both methods. The majorant rates are of the same order.

\begin{table}[htbp]
\centering
\caption{L-shaped domain, graded anisotropic meshes: conforming and CR
$L^2$-errors and floating-point evaluations of the majorant formulas.}
\label{tab:lshape-graded-error}
\small
\begin{tabular}{crrrrrr}
\toprule
method & $n$ & $L^2$-error & error rate & majorant & majorant rate & effectivity\\
\midrule
$\Pone$ & 16  & 5.937344e-3 & --    & 1.669951e-2 & --    & 2.813\\
$\Pone$ & 32  & 1.513191e-3 & 2.019 & 4.565531e-3 & 1.915 & 3.017\\
$\Pone$ & 64  & 3.801758e-4 & 2.016 & 1.235124e-3 & 1.908 & 3.249\\
$\Pone$ & 128 & 9.519684e-5 & 2.009 & 3.269425e-4 & 1.928 & 3.434\\
\midrule
$\CR$ & 16  & 3.716674e-3 & --    & 2.436678e-2 & --    & 6.556\\
$\CR$ & 32  & 9.395636e-4 & 2.031 & 6.588487e-3 & 1.931 & 7.012\\
$\CR$ & 64  & 2.357421e-4 & 2.018 & 1.752040e-3 & 1.933 & 7.432\\
$\CR$ & 128 & 5.902421e-5 & 2.009 & 4.573198e-4 & 1.949 & 7.748\\
\bottomrule
\end{tabular}
\end{table}

\FloatBarrier
\subsection{Three-dimensional tetrahedral meshes}\label{subsec:3d-numerics}
The three-dimensional test uses the unit cube
$\Omega=(0,1)^3$ and
\begin{align}\label{eq:cube-solution}
  u(x,y,z)=\sin(\pi x)\sin(\pi y)\sin(\pi z),
  \quad
  f=3\pi^2u.
\end{align}
Cartesian boxes are divided into six face-compatible Freudenthal tetrahedra. For the box indexed by $(i,j,k)$, we set
\begin{align*}
b_z:=\mathbf 1_{\{k=n_z\}}, \quad b_x:=\mathbf 1_{\{i=n_x\}}\oplus b_z, \quad b_y:=\mathbf 1_{\{j=n_y\}}\oplus b_z,
\end{align*}
where $\oplus$ denotes addition modulo two. The standard Freudenthal split is reflected in the local $x$- and $y$-midplanes according to $b_x$ and $b_y$. This preserves face compatibility and gives each tetrahedron an interior vertex without adding mesh vertices. We use the CR face basis $\psi_i=1-3\lambda_i$ and $d^{-2}=1/9$ in \eqref{eq:defect-decomp}. Smooth integrals are computed with a sixth-order tensor Gauss--Duffy rule; increasing the order to eight does not change the observed convergence rates. We compare the uniform family $n_x=n_y=n_z=n$ with the stretched family
\begin{align*}
  n_x=n_y=n,\quad n_z=n^2,
\end{align*}
so that the aspect ratio grows with $n$ while $h=O(n^{-1})$.  We use $\mathrm{AR}_h=\max_T h_T/a_{T,\min}$ and show representative boundary triangulations in Figure~\ref{fig:mesh-families-3d}.

\begin{figure}[htbp]
\centering
\includegraphics[width=\textwidth]{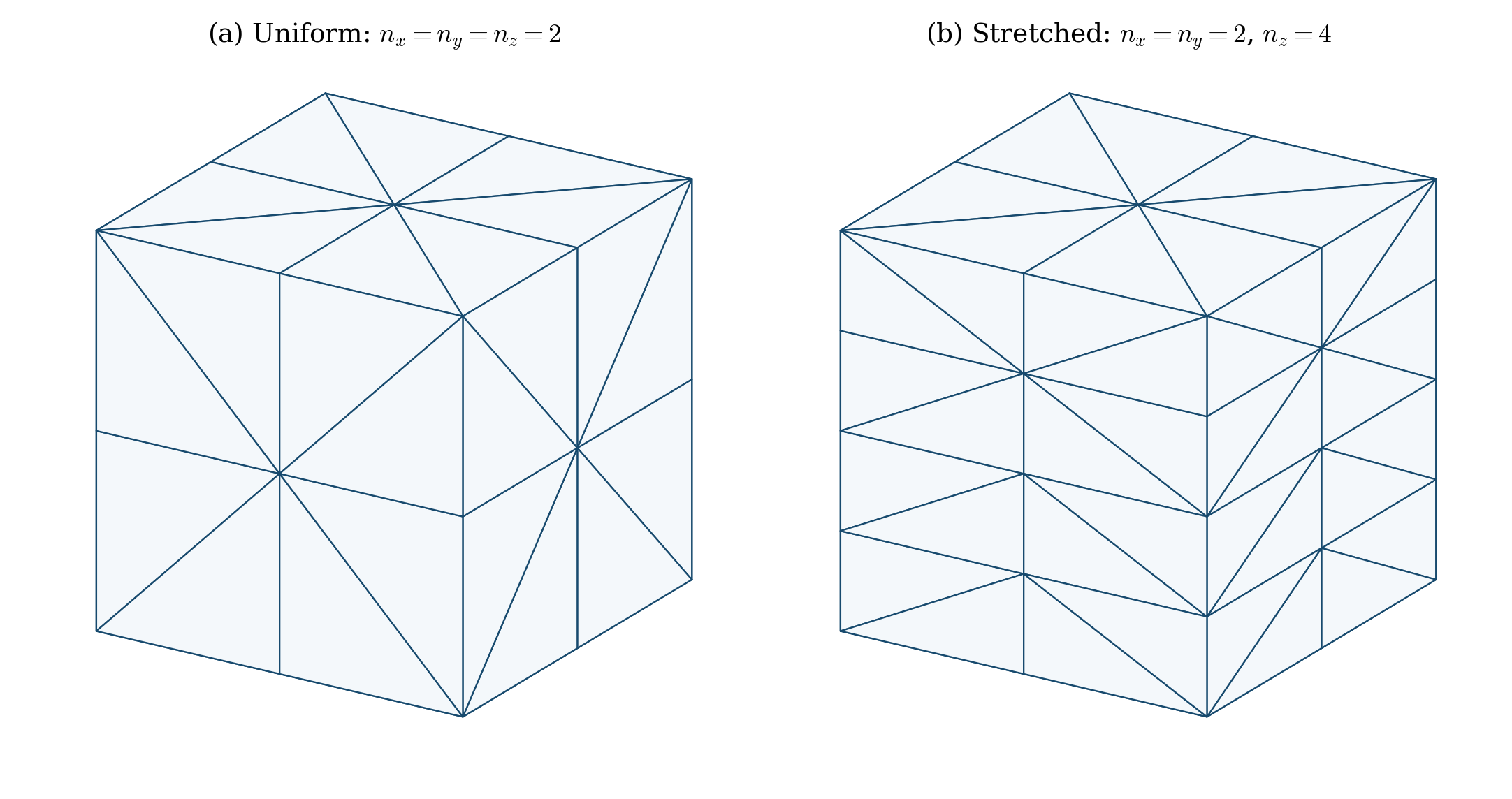}
\caption{Boundary triangulations of representative tetrahedral meshes of the unit cube: (a) the uniform family with $n_x=n_y=n_z=2$ and (b) the stretched family with $n_x=n_y=2$ and $n_z=4=n_x^2$.  Each Cartesian cell is subdivided into six Freudenthal tetrahedra using the face-compatible body-diagonal rule described in the text.  No additional vertices are introduced, and every tetrahedron has at least one interior vertex.}
\label{fig:mesh-families-3d}
\end{figure}

For both three-dimensional families, Table~\ref{tab:3d-constants} shows an almost constant ratio $\gamma_{\rm geo,h}/h$. The CR--conforming gap is again the larger term.

\begin{table}[htbp]
\centering
\caption{Three-dimensional tetrahedral meshes: spectral decomposition.  The
uniform family has $n_z=n$, whereas the stretched family has $n_z=n^2$.}
\label{tab:3d-constants}
\small
\begin{tabular}{crrrrrrr}
\toprule
family & $n$ & $N_T$ & $h$ & $\mathrm{AR}_h$ &
$\gamma_{\rm gap,h}$ & $\gamma_{\rm geo,h}$ & $\kappa_{M,h}$\\
\midrule
uniform   & 2 & 48   & 0.866025 & 2.45 & 0.128116 & 0.058926 & 0.141017\\
uniform   & 4 & 384  & 0.433013 & 2.45 & 0.099391 & 0.029463 & 0.103666\\
uniform   & 6 & 1296 & 0.288675 & 2.45 & 0.074027 & 0.019642 & 0.076588\\
uniform   & 8 & 3072 & 0.216506 & 2.45 & 0.059704 & 0.014731 & 0.061495\\
uniform   & 12 & 10368 & 0.144338 & 2.45 & 0.042783 & 0.009821 & 0.043896\\
\midrule
stretched & 2 & 96   & 0.750000 & 3.35 & 0.127113 & 0.051875 & 0.136987\\
stretched & 4 & 1536 & 0.359035 & 5.92 & 0.087649 & 0.024978 & 0.091005\\
stretched & 6 & 7776 & 0.237333 & 8.66 & 0.064877 & 0.016531 & 0.066871\\
stretched & 8 & 24576 & 0.177466 & 11.45 & 0.051773 & 0.012366 & 0.053175\\
stretched & 12 & 124416 & 0.118056 & 17.06 & 0.036697 & 0.008229 & 0.037574\\
\bottomrule
\end{tabular}
\end{table}

The conforming $\Pone$ error in Table~\ref{tab:3d-error} is still pre-asymptotic on the coarse levels, but its rate approaches two on the finer meshes in both families. The CR rates settle near two earlier. These computations check the scalar decomposition and the majorant on tetrahedral meshes; no three-dimensional anisotropic rate theorem is claimed.

\begin{table}[htbp]
\centering
\caption{Three-dimensional tetrahedral meshes: conforming and CR $L^2$-errors
and floating-point evaluations of the majorant formulas.}
\label{tab:3d-error}
\small
\begin{tabular}{ccrrrrrr}
\toprule
family & method & $n$ & $L^2$-error & error rate & majorant & majorant rate & effectivity\\
\midrule
uniform & $\Pone$ & 2 & 1.148097e-1 & --    & 9.444788e-1 & --    & 8.226\\
uniform & $\Pone$ & 4 & 8.180157e-2 & 0.489 & 2.506529e-1 & 1.914 & 3.064\\
uniform & $\Pone$ & 6 & 4.149342e-2 & 1.674 & 1.155947e-1 & 1.909 & 2.786\\
uniform & $\Pone$ & 8 & 2.435139e-2 & 1.853 & 6.812109e-2 & 1.838 & 2.797\\
uniform & $\Pone$ & 12 & 1.114433e-2 & 1.928 & 3.189735e-2 & 1.871 & 2.862\\
\addlinespace
uniform & $\CR$ & 2 & 1.310160e-1 & --    & 1.003666e+0 & --    & 7.661\\
uniform & $\CR$ & 4 & 2.951418e-2 & 2.150 & 3.113736e-1 & 1.689 & 10.550\\
uniform & $\CR$ & 6 & 1.321750e-2 & 1.981 & 1.497930e-1 & 1.805 & 11.333\\
uniform & $\CR$ & 8 & 7.477585e-3 & 1.980 & 8.930550e-2 & 1.798 & 11.943\\
uniform & $\CR$ & 12 & 3.340442e-3 & 1.987 & 4.226759e-2 & 1.845 & 12.653\\
\midrule
stretched & $\Pone$ & 2  & 1.298679e-1 & --    & 7.085632e-1 & --    & 5.456\\
stretched & $\Pone$ & 4  & 5.985543e-2 & 1.051 & 1.761123e-1 & 1.890 & 2.942\\
stretched & $\Pone$ & 6  & 2.850373e-2 & 1.792 & 8.194374e-2 & 1.848 & 2.875\\
stretched & $\Pone$ & 8  & 1.636391e-2 & 1.909 & 4.811709e-2 & 1.832 & 2.940\\
stretched & $\Pone$ & 12 & 7.371805e-3 & 1.956 & 2.242286e-2 & 1.873 & 3.042\\
\addlinespace
stretched & $\CR$ & 2  & 8.887260e-2 & --    & 7.930636e-1 & --    & 8.924\\
stretched & $\CR$ & 4  & 1.913691e-2 & 2.085 & 2.234236e-1 & 1.720 & 11.675\\
stretched & $\CR$ & 6  & 8.495271e-3 & 1.962 & 1.065552e-1 & 1.789 & 12.543\\
stretched & $\CR$ & 8  & 4.778049e-3 & 1.980 & 6.300409e-2 & 1.808 & 13.186\\
stretched & $\CR$ & 12 & 2.122737e-3 & 1.990 & 2.956171e-2 & 1.856 & 13.926\\
\bottomrule
\end{tabular}
\end{table}

In an additional implementation check with a random piecewise constant right-hand side, the residuals in the interior normal-flux continuity condition and in \eqref{eq:defect-decomp} were at roundoff level on representative uniform and stretched meshes.

\section{Concluding remarks}\label{sec:conclusion}
For the lowest-order CR--RT pair, the Marini relation reduces the equilibrated RT reconstruction to a scalar CR solve. The defect splits exactly into the CR--conforming energy gap and an explicit geometric term, and $\kappa_{M,h}^2$ is obtained from a scalar generalised eigenvalue problem. Together with the local Poincar\'e constant, this gives conforming $\Pone$ and CR error bounds without global $H^2$-regularity. On the graded L-shaped family, $q>3/2$ gives $\kappa_{M,h}=O(h)$ and therefore computable $O(h^2)$ $L^2$-operator bounds.

Sharp rates for graded tetrahedral meshes with corner and edge singularities are still open. The numerical eigenvalue computations here use standard floating-point arithmetic; certified majorants would require verified enclosures.

\section*{Data availability statement}
No external datasets were used.  All numerical data presented in this article
are contained in the manuscript.

\end{document}